\documentclass[11pt,reqno]{amsart}

\usepackage[T1]{fontenc}
\usepackage[utf8]{inputenc}
\usepackage{lmodern}
\usepackage{microtype}
\usepackage{mathtools,amssymb,mathrsfs}
\usepackage{amsthm}
\usepackage{aliascnt}
\usepackage{enumitem}
\usepackage{etoolbox}
\usepackage{xcolor}
\usepackage[margin=1.15in]{geometry}
\usepackage[colorlinks=true,linkcolor=blue!55!black,citecolor=green!40!black,urlcolor=blue!65!black]{hyperref}
\usepackage[nameinlink,capitalise,noabbrev]{cleveref}
\numberwithin{equation}{section}
\setlist[enumerate]{leftmargin=2.1em,itemsep=2pt,topsep=4pt}
\AtBeginEnvironment{thebibliography}{\setlength{\itemsep}{0pt}\setlength{\parsep}{0pt}}

\newtheorem{theorem}{Theorem}[section]
\newaliascnt{proposition}{theorem}
\newtheorem{proposition}[proposition]{Proposition}
\aliascntresetthe{proposition}
\newaliascnt{lemma}{theorem}
\newtheorem{lemma}[lemma]{Lemma}
\aliascntresetthe{lemma}
\newaliascnt{corollary}{theorem}
\newtheorem{corollary}[corollary]{Corollary}
\aliascntresetthe{corollary}
\newaliascnt{criterion}{theorem}
\newtheorem{criterion}[criterion]{Criterion}
\aliascntresetthe{criterion}
\theoremstyle{definition}
\newaliascnt{definition}{theorem}
\newtheorem{definition}[definition]{Definition}
\aliascntresetthe{definition}
\newaliascnt{convention}{theorem}
\newtheorem{convention}[convention]{Convention}
\aliascntresetthe{convention}
\theoremstyle{remark}
\newaliascnt{remark}{theorem}
\newtheorem{remark}[remark]{Remark}
\aliascntresetthe{remark}
\crefname{proposition}{proposition}{propositions}
\crefname{lemma}{lemma}{lemmas}
\crefname{corollary}{corollary}{corollaries}
\crefname{criterion}{criterion}{criteria}
\crefname{definition}{definition}{definitions}
\crefname{convention}{convention}{conventions}
\crefname{remark}{remark}{remarks}
\newtheorem*{theoremA}{Theorem A (Affine exceptional-cluster Koszul resolution)}
\newtheorem*{theoremB}{Theorem B (Common finite cores without lattice divisibility)}
\newtheorem*{theoremC}{Theorem C (Affine fibre theorem)}

\newcommand{\kk}{\Bbbk}
\newcommand{\HH}{\mathcal H}
\newcommand{\CC}{\mathcal C}
\newcommand{\WW}{\mathcal W}
\newcommand{\VV}{\mathcal V}
\newcommand{\TT}{\mathcal T}
\newcommand{\FF}{\mathcal F}
\newcommand{\Mdual}{M_c}
\newcommand{\Adual}{A_c}
\newcommand{\Pcox}{P_c}
\newcommand{\re}{\mathrm{re}}
\newcommand{\Ext}{\operatorname{Ext}}
\newcommand{\Hom}{\operatorname{Hom}}
\newcommand{\End}{\operatorname{End}}
\newcommand{\Tor}{\operatorname{Tor}}
\newcommand{\pd}{\operatorname{pd}}
\newcommand{\rad}{\operatorname{rad}}
\newcommand{\relint}{\operatorname{relint}}
\newcommand{\Cone}{\operatorname{Cone}}
\newcommand{\Supp}{\operatorname{Supp}}
\newcommand{\wide}{\operatorname{wide}}
\newcommand{\cox}{\operatorname{cox}}
\newcommand{\Fr}{\operatorname{Fr}}
\newcommand{\add}{\operatorname{add}}
\newcommand{\ind}{\operatorname{ind}}
\newcommand{\Exc}{\operatorname{Exc}}
\newcommand{\eWide}{\operatorname{eWide}}
\newcommand{\rank}{\operatorname{rank}}
\newcommand{\tauc}{\tau_c}
\newcommand{\tauD}{\tau_D}

\newcommand{\DeltaReal}{\Delta_c^{\re}(\Phi)}
\newcommand{\DeltaExc}{\Delta_c^{\mathrm{exc}}(\HH)}
\newcommand{\DivR}{\operatorname{Div}_{R}}

\title[Affine dual braid monoids and Koszul resolutions]{Affine Dual Braid Monoids:\\ Finite Cores, Exceptional Cluster Complexes, and Koszul Resolutions}
\author{Jindong Yan and Shenglin Zhu}
\address{School of Mathematical Sciences, Fudan University, Shanghai, China}
\date{August 2026}
\hypersetup{
  pdftitle={Affine Dual Braid Monoids: Finite Cores, Exceptional Cluster Complexes, and Koszul Resolutions},
  pdfauthor={Jindong Yan and Shenglin Zhu},
  pdfsubject={Koszul resolutions for crystallographic affine dual braid monoid algebras}
}

\begin{document}

\begin{abstract}
For every finite-rank crystallographic affine Coxeter system $(W,S)$ and Coxeter element $c$, we construct a minimal linear graded free resolution of the trivial module over $\kk[M([1,c]_T)]$ supported on a rectified exceptional cluster complex. Hence the affine dual braid monoid algebra is Koszul over every field $\kk$.

The exceptional complex is introduced to recover the principal-fibre topology missing from the direct Reading--Stella labelling. Half-orbit rectification replaces the transjective root labels by ordinary exceptional modules, so that a face $F$ determines an exceptional wide subcategory and the intrinsic weight
\[
  \omega(F)=\operatorname{cox}(\operatorname{wide}\langle F\rangle).
\]
The resulting principal fibres are induced subcomplexes and split canonically as joins of subcomplexes attached to connected Dynkin and affine blocks; these subcomplexes are contractible.

Affine non-lattice divisibility creates the genuinely nonprincipal case. The McCammond--Sulway completion shows that whenever no greatest interval right divisor exists, all maximal interval right divisors share a common nontrivial complete finite Coxeter component. In the associated exceptional-wide decompositions, this common Coxeter component is the Coxeter element of a Dynkin block, and the subcomplex attached to that block occurs as a common contractible join factor. Thus every nonidentity fibre is contractible, and the weighted-face complex is exact, minimal and linear. In particular, $\operatorname{Tor}^{A_c}_q(\kk,\kk)$ is indexed by $q$-vertex exceptional cluster faces in internal degree $q$, and $\operatorname{pd}_{A_c}\kk=|S|$.
\end{abstract}

\maketitle
\tableofcontents

\section{Introduction}

\subsection{The affine cluster--Koszul problem}
Let $(W,S)$ be a Coxeter system, let $T$ be its reflection set, and let $c$ be a Coxeter element. The absolute-order interval
\[
  P_c=[1,c]_T
\]
defines the dual braid monoid $M(P_c)$. In finite Coxeter type, Bessis proved that $P_c$ is a lattice and that $M(P_c)$ is a Garside monoid embedded in the corresponding Artin group \cite[Fact~2.3.1, Theorem~2.3.2 and Corollary~2.3.3]{Bessis2003}.

Josuat--Verg\`es and Nadeau gave this finite theory a homological form \cite{JVN2023}. Their resolution is supported on the positive cluster complex and is controlled, degree by degree, by divisibility fibres. If cluster faces are weighted by interval elements, then the homogeneous strand of the weighted-face complex is the augmented simplicial chain complex of
\[
  K_b=\{F:\omega(F)\preccurlyeq_R b\}.
\]
Contractibility of these fibres gives exactness, while the face grading gives linearity and minimality. In finite type, lattice divisibility reduces the relevant fibres to positive cluster balls. Josuat--Verg\`es and Nadeau therefore asked whether this cluster-complex mechanism extends to affine dual braid monoids \cite[Section~9.3]{JVN2023}.

In affine type the same question is best formulated at the level of fibres. The real Reading--Stella complex supplies the appropriate affine cluster geometry, but its natural labelling does not directly provide the exceptional-wide principal fibres required by the weighted resolution. At the same time, $P_c$ need not be a lattice, so a general monoid element may have several maximal interval right divisors and no greatest one. The two new structures of this paper enter at these different levels: a rectified exceptional cluster complex makes principal fibres intrinsic and decomposable, while common finite Coxeter cores replace the greatest-divisor argument when lattice divisibility fails.

The proof combines the crystallographic completion of McCammond--Sulway with the real affine cluster geometry of Reading--Stella and the exceptional-wide correspondence of Hanson--Reading, within the weighted resolution framework of Josuat--Verg\`es and Nadeau. The new ingredient is a fibre-level bridge between these structures: positive saturation produces common finite Coxeter components for maximal interval divisors, and rectification identifies them with Dynkin blocks that occur as contractible join factors.

\subsection{Main theorem}
Put
\[
  \Mdual=M(P_c),\qquad \Adual=\kk[\Mdual].
\]
Fix a hereditary realization $\HH$ of the Coxeter datum $(W,S,c)$, chosen componentwise so that the ordered simple objects realize the prescribed Coxeter element $c$, as in \cref{sec:coxeter-hereditary}. This realization is auxiliary and is used only to construct the exceptional complex and its intrinsic face weights. Let $\DeltaExc$ denote the rectified exceptional $c$-cluster complex of \cref{def:rectified}. For a simplicial complex $K$, let $\Fr_q(K)$ denote the set of its $q$-vertex faces.

\begin{theoremA}
Let $(W,S)$ be a finite direct product of irreducible crystallographic affine Coxeter systems. For every Coxeter element $c$ and every field $\kk$, the algebra $\Adual$ has an explicit minimal linear graded free resolution by free left $\Adual$-modules
\[
  P_q=\Adual\otimes_{\kk}\kk\Fr_q(\DeltaExc).
\]
Consequently $\Adual$ is Koszul,
\[
  \Tor_q^{\Adual}(\kk,\kk)\cong \kk\Fr_q(\DeltaExc)
\]
is concentrated in internal degree $q$, and
\[
  \pd_{\Adual}\kk=|S|.
\]
\end{theoremA}

The differential is given in \cref{thm:minimal-resolution}. By \cref{thm:rectified}, the same basis may equivalently be labelled by real Reading--Stella faces. For finite direct products, the construction is componentwise; equivalently, the resulting resolution is the total tensor product of the component resolutions.

\subsection{Exceptional cluster complexes and principal fibres}
Reading and Stella's real affine $c$-cluster complex $\DeltaReal$ provides the root-theoretic model for affine cluster compatibility \cite{ReadingStella2020}. Its natural root labelling, however, does not place all vertices in a single ordinary exceptional-module section of a hereditary category. Consequently it does not directly furnish the exceptional wide subcategories needed for an intrinsic face weight.

A half-orbit rectification removes this mismatch. For the fixed tame hereditary realization $\HH$, the real Reading--Stella complex is reconstructed as a flag complex $\DeltaExc$ on ordinary exceptional modules,
\[
  \DeltaExc\xrightarrow{\sim}\DeltaReal.
\]
Since every rectified face admits an exceptional ordering, a face $F$ determines an exceptional wide subcategory and hence an intrinsic weight
\[
  \WW_F=\wide\langle X:X\in F\rangle,
  \qquad
  \omega(F)=\cox(\WW_F)\in P_c.
\]

For $w\in P_c$, let $\WW_w=\cox^{-1}(w)$. The corresponding principal fibre is then the induced subcomplex
\[
  K_w=\DeltaExc[\ind\Exc(\WW_w)].
\]
If
\[
  \WW_w=\VV_1\oplus\cdots\oplus\VV_s
\]
is the decomposition into connected exact blocks, orthogonality of distinct blocks gives the ambient join decomposition
\[
  K_w=K_{\VV_1}*\cdots*K_{\VV_s},
  \qquad
  K_{\VV_i}=\DeltaExc[\ind\Exc(\VV_i)].
\]
Each connected block is either Dynkin or affine. In the Dynkin case, a split HRS tilt identifies $K_{\VV_i}$ with a positive finite cluster ball; in the affine case, the inherited rectification identifies it with the corresponding real affine Reading--Stella complex. The latter is contractible by Appendix~\ref{app:topology}. Hence every nontrivial principal fibre is contractible.

Thus rectification turns principal fibres into induced categorical subcomplexes whose topology is governed by connected-block decompositions. In particular, the significance of the exceptional cluster complex goes beyond the simplicial relabelling itself: it is the model in which the principal-fibre structure becomes intrinsic.

\subsection{Non-lattice divisibility and finite cores}
Principal fibres do not exhaust the affine problem. For $b\in\Mdual$,
\[
  K_b=\bigcup_{w\in\operatorname{Max}_{R,c}(b)}K_w,
\]
where $\operatorname{Max}_{R,c}(b)$ denotes the maximal interval right divisors of $b$. If a greatest interval right divisor $d$ exists, then $K_b=K_d$. The genuinely nonprincipal case occurs when no such $d$ exists, and this is precisely where the failure of lattice divisibility enters.

In rank at least three, the McCammond--Sulway crystallographic completion \cite{McCammondSulway2017} embeds the affine interval group into a larger Garside setting where completed gcds exist. We prove that positivity is saturated along the original interval group, so the completed right gcd with $c$ detects exactly the original interval right divisors. When there is no greatest original divisor, the completed structure forces a finite canonical Coxeter component common to all maximal ones. In affine rank two the original interval is already a height-two lattice, so this nonprincipal case does not occur.

\begin{theoremB}
Let $W$ be irreducible crystallographic affine. If $1\ne b\in M(P_c)$ has no greatest $P_c$-right divisor, then all maximal $P_c$-right divisors of $b$ share a common nontrivial irreducible complete finite Coxeter component.
\end{theoremB}

Here $e$ is called a complete finite Coxeter component of $w\in P_c$ if $W(e)$ is a finite irreducible connected component of Dyer's canonical Coxeter system of $W(w)$ and $e$ is its Coxeter element in the component factorization of $w$. We refer informally to a common such component as a \emph{finite core}. The same divisor analysis also proves that the natural map $M(P_c)\to G_W$ is injective, hence that the interval monoid is cancellative.

\subsection{From finite cores to contractible fibres}
The same connected-block decomposition also governs the Coxeter side. For
\[
  \WW_w=\VV_1\oplus\cdots\oplus\VV_s,
  \qquad e_i=\cox(\VV_i),
\]
let
\[
  W(w)=\langle t\in T:t\le_T w\rangle.
\]
Then
\[
  W(w)=W(e_1)\times\cdots\times W(e_s),
\]
and $e_i$ is an irreducible complete finite Coxeter component of $w$ exactly when $\VV_i$ is a Dynkin block. If $e=e_i$ is such a component and $u$ is the product of the remaining component Coxeter elements, then
\[
  w=eu=ue,
  \qquad
  K_w=K_e*K_u
\]
inside the ambient exceptional cluster complex. The subcomplex $K_e$ is the positive finite cluster ball associated with the Dynkin block and is therefore nonempty and contractible.

Now suppose that $b$ has no greatest interval right divisor. By Theorem~B, its maximal interval right divisors share a complete finite Coxeter component $e$. For each maximal divisor $w$, the preceding decomposition gives $K_w=K_e*K_{u_w}$, and, since all these joins occur inside the same ambient complex,
\[
  K_b
  =K_e*\left(\bigcup_w K_{u_w}\right).
\]
Hence $K_b$ is contractible. Together with the principal case, this yields the fibre theorem.

\begin{theoremC}
For every $1\ne b\in M(P_c)$, the fibre
\[
  K_b=\{F\in\DeltaExc:\omega(F)\preccurlyeq_R b\}
\]
is contractible. More precisely, either $b$ has a greatest interval right divisor $d$ and $K_b=K_d$, or there is a common nontrivial complete finite Coxeter component $e$ such that
\[
  K_b=K_e*L_b
\]
for a simplicial subcomplex $L_b$.
\end{theoremC}

The fine homogeneous degree-$b$ strand of the weighted-face complex is the augmented simplicial chain complex of $K_b$. By Theorem~C these strands are exact in every nonzero degree; the grading gives linearity, and the positive-degree differential coefficients give minimality. Hence the weighted-face complex is the minimal linear resolution of Theorem~A.

The paper is organized as follows. Section~\ref{sec:coxeter-hereditary} fixes the Coxeter and hereditary conventions. Sections~\ref{sec:completion}--\ref{sec:common-embedding} establish the finite-core divisor theory, Section~\ref{sec:rectification} constructs the rectified exceptional cluster complex, and Sections~\ref{sec:weights}--\ref{sec:fibres} establish the fibre theory. Section~\ref{sec:weighted} constructs the minimal linear resolution, and Appendix~\ref{app:topology} proves the real affine cluster contractibility used for affine principal factors.

\section{Coxeter intervals and hereditary realizations}\label{sec:coxeter-hereditary}

\subsection{Absolute order and interval monoids}
Let $(W,S)$ have finite rank $n=|S|$ and reflection set
\[
  T=\{wsw^{-1}:w\in W,\ s\in S\}.
\]
The reflection length is $\ell_T(w)=\min\{r:w=t_1\cdots t_r,\ t_i\in T\}$. The absolute order is
\begin{equation}\label{eq:absolute-order}
  u\le_Tv\iff \ell_T(v)=\ell_T(u)+\ell_T(u^{-1}v),
\end{equation}
and equivalently
\begin{equation}\label{eq:absolute-right}
  u\le_Tv\iff \ell_T(v)=\ell_T(vu^{-1})+\ell_T(u).
\end{equation}
Fix $c$ and put $\Pcox=[1,c]_T$. The interval monoid $\Mdual=M(\Pcox)$ is generated by symbols $[w]$, $w\in\Pcox\setminus\{1\}$, with relations $[u][v]=[uv]$ whenever $u,v,uv\in\Pcox$ and reflection length is additive. We suppress brackets. The grading is $\deg w=\ell_T(w)$. Right divisibility is
\[
  u\preccurlyeq_Rv\iff v=au\quad\text{for some }a.
\]
For a balanced interval $P=[1,\Delta]$, write
\[
  G(P)^+=\langle P\rangle^+\subseteq G(P)
\]
for the positive submonoid of its interval group. We do not identify $G(P)^+$ with $M(P)$ until the natural map $M(P)\to G(P)$ is shown to be injective.

\begin{lemma}[Interval degree]\label{lem:interval-degree}
The homogeneous interval presentation induces degree homomorphisms
\[
  \deg:M(P)\longrightarrow\mathbb N,
  \qquad
  \deg:G(P)\longrightarrow\mathbb Z,
\]
whose restriction to $P$ is reflection length. Moreover, $M(P)$ is generated by its degree-one interval elements, namely the reflections in $P$.
\end{lemma}

\begin{proof}
Every defining interval relation $uv=w$ is reflection-length additive, so the presentation is homogeneous and the degree extends to the universal monoid and universal group. If $w\in P$ and $w=t_1\cdots t_r$ is a reduced reflection factorization, each prefix $t_1\cdots t_j$ lies below $w$ in absolute order and hence belongs to $P$. The successive products are therefore defining interval relations, so the monoid generator $w$ equals the product of the degree-one generators $t_1\cdots t_r$.
\end{proof}

\begin{convention}[Koszulity]\label{conv:koszul}
For a connected $\mathbb N$-graded algebra $A$ with $A_0=\kk$, not assumed locally finite, we call $A$ \emph{Koszul} when the augmentation module $\kk$ admits a linear graded free resolution. All free resolutions below are resolutions of left modules unless explicitly stated otherwise.
\end{convention}

\begin{convention}[Cancellation]\label{conv:cancellation}
Before the interval-monoid embedding is proved, cancellation is used only in groups or in positive monoids already known to be Garside or quasi-Garside.
\end{convention}

\subsection{Hereditary realization}
From \cref{sec:rectification} through \cref{sec:complete-components}, initially assume $(W,S)$ irreducible, crystallographic and affine. Fix a reduced Coxeter word $c=s_1\cdots s_n$.

\begin{lemma}[Compatible hereditary realization]\label{lem:hereditary-realization}
There exists a basic connected tame hereditary module category $\HH=\operatorname{mod}\Lambda_c$ realizing the prescribed crystallographic affine datum, with simple objects $S_1,\ldots,S_n$, such that the ordering $(S_1,\ldots,S_n)$ is exceptional and determines the Coxeter element $c=s_1\cdots s_n$, and the Euler form is the form denoted $E_{c^{-1}}$ in the Hanson--Reading convention.
\end{lemma}

\begin{proof}
The standard species construction for a symmetrizable Cartan datum, with the acyclic orientation induced by the order
\[
  1<\cdots<n
\]
in the prescribed expression $c=s_1\cdots s_n$, gives such a basic connected hereditary module category; see \cite[Section~2.2]{HansonReading2025}. For the corresponding simple objects $S_1,\ldots,S_n$, this orientation gives
\[
  \Ext^1(S_j,S_i)=0\qquad(j>i).
\]
Hence $(S_1,\ldots,S_n)$ is an exceptional ordering, and with the Hanson--Reading convention its Coxeter element is
\[
  s_1\cdots s_n=c.
\]
In affine type the realization is tame hereditary, and with this ordering its Euler form is the form $E_{c^{-1}}$ used in \cite[Sections~5 and 8]{HansonReading2025}.
\end{proof}

Fix one realization supplied by \cref{lem:hereditary-realization} and write
\begin{equation}\label{eq:fixed-realization}
  \HH=\operatorname{mod}\Lambda_c.
\end{equation}
For any hereditary realization under consideration, write $K$ for its ground field when linear dimensions are written explicitly. The symbol $\kk$ remains reserved for the coefficient field of $\Adual$.

Identify $K_0(\HH)$ with the affine root lattice. Write $\Pi=\{\alpha_1,\ldots,\alpha_n\}$ and $V=\mathbb R\otimes_{\mathbb Z}K_0(\HH)$. For an object $X$, write $\dim X=[X]$. If $\alpha$ is real, $t_\alpha$ is the associated reflection.

An indecomposable $X$ is exceptional if $\End(X)$ is a division algebra and $\Ext^1(X,X)=0$. An exceptional sequence $(X_1,\ldots,X_r)$ satisfies
\[
  \Hom(X_j,X_i)=0=\Ext^1(X_j,X_i)\qquad(j>i).
\]
A full subcategory $\WW\subseteq\HH$ is wide if it is exact abelian and extension-closed. Write $\wide\langle\mathscr S\rangle$ for the smallest wide subcategory containing $\mathscr S$. If $(X_1,\ldots,X_r)$ is complete in an exceptional wide subcategory $\WW$, define
\begin{equation}\label{eq:cox-wide}
  \cox(\WW)=t_{\dim X_1}\cdots t_{\dim X_r}.
\end{equation}

Let $D=\Hom_K(-,K)$, let $\tau$ be ordinary Auslander--Reiten translation, and let $\tauD$ denote the derived AR autoequivalence. The Serre functor of $D^b(\HH)$ is
\begin{equation}\label{eq:Serre}
  \mathbb S=\tauD[1].
\end{equation}
For a primitive idempotent $e_i$, put $P_i=\Lambda_ce_i$ and $I_i=D(e_i\Lambda_c)$. Then functorially in $N$,
\begin{equation}\label{eq:proj-inj-duality}
  D\Hom_{\HH}(N,I_i)\cong\Hom_{\HH}(P_i,N).
\end{equation}

\begin{theorem}[Hanson--Reading correspondences]\label{thm:HR}
The following hold \cite[Theorems~2.2, 2.4 and Propositions~4.1, 4.3]{HansonReading2025}.
\begin{enumerate}[label=\textup{(\alph*)}]
\item Complete exceptional sequences correspond bijectively to reduced reflection factorizations of $c$.
\item $\cox:\eWide(\HH)\xrightarrow{\sim}[1,c]_T$ is a rank-preserving poset isomorphism.
\item Every exceptional sequence is complete in its exceptional wide closure, and the rank of that wide closure equals the length of the sequence. In particular, an exceptional sequence is complete in $\HH$ if and only if its length is $\rank K_0(\HH)$.
\item Every exceptional wide subcategory is exact-equivalent to the module category of a finite-dimensional hereditary algebra.
\end{enumerate}
\end{theorem}

\subsection{The Reading--Stella geometric model}
Let $\Phi_c^{\re}$ be the real affine almost-positive roots of Reading--Stella and let $\DeltaReal$ be their real $c$-cluster complex. Compatibility is denoted $\alpha\sim_c\beta$. By \cite[Definition~5.3 and Theorem~5.5]{ReadingStella2020}, the complex is flag; every real facet has $n$ vertices by \cite[Proposition~5.14]{ReadingStella2020}. The Reading--Stella permutation is $\tauc$.

For a simplicial complex $K$, $K[A]$ denotes the induced subcomplex on $A$, and $K*L$ denotes the join. The notation $|K|$ always means the standard weak geometric realization.

\subsection{Products and topology}
We use ``finite-rank crystallographic affine Coxeter system'' for a finite direct product of irreducible affine systems. If $(W,c)=\prod_a(W_a,c_a)$, set $\HH=\bigoplus_a\HH_a$ and
\[
  \Delta_c^{\mathrm{exc}}(\HH)=\Delta_{c_1}^{\mathrm{exc}}(\HH_1)*\cdots*\Delta_{c_r}^{\mathrm{exc}}(\HH_r),
\]
with componentwise face weight. The irreducible proof is assembled by tensor products in \cref{thm:main}.

The weak topology of $|\DeltaReal|$ is not identified with the Euclidean topology of the fan support. Appendix~\ref{app:topology} compares them by a numerable good cover.

\section{Crystallographic completion and positive saturation}\label{sec:completion}

We first establish the divisibility and cancellation properties of the affine interval monoid. We begin by isolating the only case in which the horizontal root system is empty. Hanson--Reading note that this occurs precisely in affine rank two \cite[Section~6]{HansonReading2025}.

\begin{remark}[The rank-two affine case]\label{rem:rank-two}
If $(W,S)$ has affine rank two, write
\[
  W=\langle s_0,s_1\mid s_0^2=s_1^2=1\rangle,
  \qquad c=s_0s_1.
\]
The reflections are $r_i=c^is_0$ for $i\in\mathbb Z$. Since
\[
  r_i^{-1}c=r_{i-1}\in T,
\]
we have $\ell_T(c)=\ell_T(r_i)+\ell_T(r_i^{-1}c)=2$, so every reflection lies below $c$ in absolute order. Conversely, if $1<u<c$, then $\ell_T(u)=1$, hence $u$ is a reflection. Therefore
\[
  P_c=[1,c]_T=\{1,c\}\sqcup T.
\]
It is a balanced height-two lattice: distinct atoms have meet $1$ and join $c$. Digne's interval criterion \cite[Theorems~5.2 and 5.4]{Digne2006} therefore makes $M(P_c)$ quasi-Garside with Garside element $c$. In particular $M(P_c)$ embeds in its interval group, and every nonidentity positive element $b$ has a greatest $P_c$-right divisor, namely $\gcd_R(b,c)$. Thus the genuinely nonprincipal case considered below occurs only in rank at least three.
\end{remark}

For the rest of \cref{sec:completion,sec:common-embedding}, assume that $W$ has rank at least three. We use the Coxeter, diagonal, factorable, and crystallographic groups and intervals of McCammond--Sulway \cite[Sections~6--7]{McCammondSulway2017}, with the weighted interval convention of \cite[Definition~7.1]{McCammondSulway2017}. We denote these intervals by
\[
  P_W,\quad P_D,\quad P_F,\quad P_C,
\]
their interval groups by $G_W,G_D,G_F,G_C$, and their positive images by
\[
  G_X^+=\langle P_X\rangle^+\subseteq G_X\qquad(X=W,D,F,C).
\]
Here $P_W=P_c=[1,c]_T$. We retain their terminology of horizontal and vertical reflections and factored translations.

\begin{theorem}[Crystallographic completion package]\label{thm:MS-package}
The following properties hold.
\begin{enumerate}[label=\textup{(\alph*)}]
\item $P_C=P_W\cup P_F$ and $P_D=P_W\cap P_F$.
\item $P_F$ is a lattice, $P_C$ is a balanced lattice, and $P_F\hookrightarrow P_C$ preserves meets and joins.
\item $G_F^+$ and $G_C^+$ are Garside monoids with Garside element $c$.
\item The interval groups embed compatibly and
\[
  G_C\cong G_F*_{G_D}G_W,\qquad G_F\cap G_W=G_D\quad\text{inside }G_C.
\]
\item If the horizontal root system has irreducible components indexed by $1,\ldots,m$, then
\[
  P_F\cong P_1\times\cdots\times P_m,\qquad G_F\cong G_1\times\cdots\times G_m.
\]
\end{enumerate}
\end{theorem}

\begin{proof}
The interval identities are \cite[Lemma~7.2 and Remark~7.3]{McCammondSulway2017}; meet and join preservation follows from \cite[Lemma~8.6 and Proposition~2.15]{McCammondSulway2017}. The completed lattice and Garside assertions are \cite[Theorems~8.9 and 8.10]{McCammondSulway2017}. The pushout and amalgamated-product assertions are \cite[Proposition~9.2 and Theorem~9.6]{McCammondSulway2017}, and the product descriptions are \cite[Proposition~7.6]{McCammondSulway2017}.
\end{proof}

\begin{lemma}[Greedy restriction]\label{lem:greedy-restriction}
The inclusion $G_F\hookrightarrow G_C$ preserves left greedy normal forms. In particular, if
\[
  g=c^p a_1\cdots a_r,\qquad a_i\in P_F\setminus\{1,c\},
\]
is the $G_F$-left greedy normal form, then the same expression is the $G_C$-left greedy normal form.
\end{lemma}

\begin{proof}
A sequence of simples $a_1\cdots a_r$ is left greedy precisely when
\[
  \partial a_i\wedge_L a_{i+1}=1\qquad(1\le i<r),
\]
where $\partial a_i=a_i^{-1}c$. Since $P_F\hookrightarrow P_C$ preserves the Garside element $c$ and the relevant left meets by \cref{thm:MS-package}(b), these conditions are unchanged under the inclusion. Hence the left greedy normal form in $G_F$ remains left greedy in $G_C$; the power of $c$ is unchanged as well.
\end{proof}

Under $G_F\cong\prod_{i=1}^mG_i$, write $c=(c_1,\ldots,c_m)$. Let
\[
  \nu_i:G_i\longrightarrow\mathbb Z
\]
be the winding homomorphism of \cite[Definition~3.8]{McCammondSulway2017}, normalized so that horizontal reflection atoms have winding $0$, factored-translation atoms have winding $1$, and $\nu_i(c_i)=1$. Put $\nu=(\nu_1,\ldots,\nu_m)$.

\begin{lemma}[Diagonal subgroup]\label{lem:diagonal}
Inside $G_F$,
\[
  G_D=\nu^{-1}\bigl(\mathbb Z(1,\ldots,1)\bigr).
\]
\end{lemma}

\begin{proof}
Put
\[
  G_H=\prod_{i=1}^m\ker\nu_i\subseteq G_F.
\]
By McCammond--Sulway's description of the winding homomorphisms and the factor decomposition, $G_H$ is the full horizontal subgroup of $G_F$; moreover their diagonal construction gives
\[
  G_D=\langle G_H,c\rangle
\]
inside $G_F$ \cite[Definition~3.8, Proposition~7.6 and Remark~7.3]{McCammondSulway2017}. Since
\[
  \nu(G_H)=0,\qquad \nu(c)=(1,\ldots,1),
\]
it follows that
\[
  G_D\subseteq\nu^{-1}\bigl(\mathbb Z(1,\ldots,1)\bigr).
\]
Conversely, if $\nu(g)=r(1,\ldots,1)$, then $\nu(gc^{-r})=0$, so $gc^{-r}\in G_H\subseteq G_D$. Since $c^r\in G_D$, we obtain $g\in G_D$.
\end{proof}

\begin{theorem}[Completed factor-word constraints]\label{thm:factor-word}
Assume $m>1$.
\begin{enumerate}[label=\textup{(\roman*)}]
\item A reduced completed factorization of $c$ cannot contain both a vertical reflection and a factored translation.
\item Letters from different horizontal components are compatible in the factor interval.
\item Two factored translations from the same horizontal component are incompatible.
\item An element of $P_F$ belongs to $P_W$ if and only if a reduced factor word contains either no factored translation or exactly one from each horizontal component.
\end{enumerate}
\end{theorem}

\begin{proof}
Assertion~(i) is \cite[Proposition~6.6]{HansonReading2025}, while (ii)--(iii) are the compatibility constraints of \cite[Proposition~6.8]{HansonReading2025}. By \cite[the paragraph immediately following Proposition~6.8]{HansonReading2025}, an element of $P_F$ belongs to $P_W$ precisely when a reduced factor word contains either no factored translation or exactly $m$ factored translations. In the latter case, (iii) forbids two of these letters from the same horizontal component. Since there are exactly $m$ horizontal components, the $m$ factored translations consist of exactly one from each component. Conversely, the same criterion shows that either factor pattern represents an element of $P_W$. This proves (iv), including the ``one from each component'' clause used below.
\end{proof}

When $m=1$, McCammond--Sulway's completion has $G_D=G_F$ and $G_W=G_C$; the synchronization arguments below are unnecessary.

\begin{lemma}[Factor-cone saturation]\label{lem:factor-cone}
Inside $G_C$,
\[
  G_C^+\cap G_F=G_F^+.
\]
\end{lemma}

\begin{proof}
Write the $G_F$-left greedy normal form of $g\in G_F$ as
\[
  g=c^pa_1\cdots a_r,\qquad a_i\in P_F\setminus\{1,c\}.
\]
By \cref{lem:greedy-restriction}, it is also the $G_C$-normal form. If $g\in G_C^+$, its $G_C$-infimum is nonnegative, so $p\ge0$ and $g\in G_F^+$.
\end{proof}

\begin{proposition}[Exact diagonal positivity]\label{prop:diagonal-positive}
Inside $G_F$,
\[
  G_F^+\cap G_D=G_D^+.
\]
\end{proposition}

\begin{proof}
For $m=1$ this is $G_D=G_F$. Assume $m>1$ and let $x=(x_1,\ldots,x_m)\in G_F^+\cap G_D$. By \cref{lem:diagonal}, all coordinate windings equal some $r\ge0$. Choose positive atom words and cut the $i$th word at its $r$ factored translations:
\[
  x_i=h_{i,0}f_{i,1}h_{i,1}\cdots f_{i,r}h_{i,r}.
\]
Letters from distinct components commute, hence
\[
  x=H_0T_1H_1\cdots T_rH_r,
  \qquad H_j=\prod_i h_{i,j},\quad T_j=\prod_i f_{i,j}.
\]
Every horizontal reflection atom has a one-letter factor word with no factored translation, so \cref{thm:factor-word}(iv) puts it in $P_W\cap P_F=P_D$. Thus $H_j\in G_D^+$. For $T_j$, the factors $f_{i,j}$ lie in distinct horizontal components; by \cref{thm:factor-word}(ii) and the direct-product structure of $P_F$, their product is a simple element of $P_F$, and its factor word is reduced with exactly one factored translation from each horizontal component. Hence \cref{thm:factor-word}(iv) gives $T_j\in P_W\cap P_F=P_D$. Therefore $x\in G_D^+$.
\end{proof}

Together with \cref{lem:factor-cone}, this also gives
\[
  G_W^+\cap G_D=G_D^+.
\]

\begin{lemma}[Positive amalgam normal form]\label{lem:positive-amalgam}
Suppose
\[
  G_C=G_F*_{G_D}G_W
\]
and
\[
  G_F^+\cap G_D=G_D^+=G_W^+\cap G_D.
\]
Let
\[
  z=x_1\cdots x_r
\]
be a positive alternating expression with nonidentity syllables in $G_F^+$ and $G_W^+$, chosen with the minimal possible number of syllables. If $r>1$, then $x_i\notin G_D$ for every $i$. Hence the expression is reduced in the amalgamated product. In particular, a reduced positive alternating word of length at least two cannot represent an element of either vertex group.
\end{lemma}

\begin{proof}
Suppose $r>1$ and $x_i\in G_D$. By the intersection hypothesis, $x_i\in G_D^+$. It may therefore be moved into an adjacent vertex group and absorbed; if $x_i$ is interior, the two neighboring syllables then lie in the same vertex group and can be combined. In either case the number of syllables decreases, contradicting minimality. Thus every syllable lies outside $G_D$, so the expression is reduced in $G_F*_{G_D}G_W$. The normal-form theorem \cite[I.1.2, Theorem~6]{Serre2003} gives the final assertion.
\end{proof}

\begin{theorem}[Positive-cone saturation]\label{thm:positive-saturation}
Let $W$ be irreducible affine of rank at least three. Then
\[
  G_C^+\cap G_W=G_W^+.
\]
\end{theorem}

\begin{proof}
If $m=1$, then $G_W=G_C$. Assume $m>1$ and let $x\in G_C^+\cap G_W$. Since $P_C=P_W\cup P_F$, choose a positive completed expression for $x$ with the least possible number of alternating $G_W^+$- and $G_F^+$-syllables. By \cref{lem:positive-amalgam}, an expression with at least two syllables is reduced in $G_F*_{G_D}G_W$ and therefore cannot represent the element $x\in G_W$. Hence there is one syllable. If it lies in $G_W^+$ we are done; if it lies in $G_F^+$, then it belongs to $G_F\cap G_W=G_D$, and \cref{prop:diagonal-positive} places it in $G_D^+\subseteq G_W^+$.
\end{proof}

\section{Common finite Coxeter components and interval-monoid embedding}\label{sec:common-embedding}

For the higher-rank case, $G_C^+$ is Garside, so every $b\in G_W^+$ has a completed greatest common right divisor with $c$:
\begin{equation}\label{eq:completed-gcd}
  d_C(b)=\gcd_R^{G_C^+}(b,c)\in P_C.
\end{equation}
Put
\[
  \DivR^{P_W}(b)=\{u\in P_W:u\preccurlyeq_Rb\text{ in }G_W^+\}.
\]

\begin{proposition}[Exact interval-divisor detection]\label{prop:divisor-detection}
For $b\in G_W^+$,
\[
  \DivR^{P_W}(b)=\{u\in P_W:u\preccurlyeq_Rd_C(b)\text{ in }G_C^+\}.
\]
\end{proposition}

\begin{proof}
If $u\in P_W$ right divides both $b$ and $c$, it right divides their completed gcd. Conversely, if $u\preccurlyeq_Rd_C(b)$, then $bu^{-1}\in G_C^+$. Since $b,u\in G_W$, the same element lies in $G_W$, and \cref{thm:positive-saturation} gives $bu^{-1}\in G_W^+$.
\end{proof}

\begin{lemma}[Diagonal lifting]\label{lem:diagonal-lifting}
If $d=d_C(b)\in P_F\setminus P_D$, there is $x\in G_D^+$ such that
\[
  d\preccurlyeq_Rx\preccurlyeq_Rb,\qquad \gcd_R^{G_C^+}(x,c)=d.
\]
\end{lemma}

\begin{proof}
We first note the elementary equality $P_F\cap G_D=P_D$. Indeed, if $s\in P_F\cap G_D$, then \cref{prop:diagonal-positive} gives $s\in G_D^+$. Writing $c=qs$ in $G_F^+$, one has $q\in G_D$, hence $q\in G_D^+$ by the same proposition; thus $s$ is a simple right divisor of $c$ in $G_D^+$ and belongs to $P_D$.

Write $b=ad$ in $G_C^+$ and choose for $a$ a positive alternating $G_W^+/G_F^+$ expression with minimal syllable number. It is nonempty, for otherwise $b=d\in G_W\cap G_F=G_D$, contradicting $d\notin P_D$. If the expression has more than one syllable, \cref{lem:positive-amalgam} makes it reduced and places every syllable outside $G_D$; if it has one syllable, that syllable cannot lie in $G_D$, since then $a,d\in G_F$ would give $b\in G_F\cap G_W=G_D$ and hence $d\in P_F\cap G_D=P_D$.

The final syllable cannot lie in $G_W^+\setminus G_D$. Indeed, appending the syllable $d\in G_F\setminus G_D$ would then produce a reduced amalgam word of length at least two representing $b\in G_W$, contrary to the normal-form theorem. Thus the final syllable is $f\in G_F^+\setminus G_D$.

Set $x=fd\in G_F^+$. We claim $x\in G_D$. If the expression for $a$ consists only of $f$, then $b=x\in G_F\cap G_W=G_D$. Otherwise the preceding syllable lies in $G_W\setminus G_D$; if $x\notin G_D$, replacing the terminal pair $fd$ by the single $G_F$-syllable $x$ gives a reduced alternating word of length at least two for $b$, again contradicting $b\in G_W$. Hence $x\in G_D$, and \cref{prop:diagonal-positive} gives $x\in G_D^+$. Any positive common right divisor of $x$ and $c$ also right divides $b$, so the defining greatest-common-divisor property of $d=d_C(b)$ gives the gcd assertion.
\end{proof}

If $d_C(b)\in P_W$, then it is the greatest $P_W$-right divisor of $b$. Hence assume $d_C(b)\notin P_W$ and put
\[
  d=d_C(b).
\]
Since
\[
  P_C=P_W\cup P_F,\qquad P_D=P_W\cap P_F,
\]
we have
\[
  d=(d_1,\ldots,d_m)\in P_F\setminus P_D.
\]
By \cref{thm:factor-word}(iii), a reduced factor word for $d$ contains at most one factored translation from each horizontal component. Hence
\[
  \nu_i(d_i)\in\{0,1\}\qquad(1\le i\le m).
\]
Define
\begin{equation}\label{eq:IZ}
  I(d)=\{i:\nu_i(d_i)=1\},\qquad Z(d)=\{i:\nu_i(d_i)=0\}.
\end{equation}
Both sets are nonempty. Indeed, if one were empty, all coordinate windings would be equal; \cref{lem:diagonal} together with $P_F\cap G_D=P_D$ would then imply $d\in P_D$, a contradiction.

\begin{lemma}[Location of interval right divisors]\label{lem:location}
If $u\in P_W$ and $u\preccurlyeq_Rd$ for $d\in P_F\setminus P_D$, then
\[
  u\in P_D,\qquad \nu(u)=(0,\ldots,0).
\]
\end{lemma}

\begin{proof}
Suppose $u\notin P_F$. Then every reduced ordinary reflection factorization of $u$ contains a vertical reflection; otherwise an entirely horizontal factorization would place $u$ in $G_F$, and factor-cone saturation would imply $u\in P_F$. Choose such a factorization. Since
\[
  u\preccurlyeq_R d\preccurlyeq_R c
\]
in the completed interval, extend it to a reduced completed factorization of $c$. This factorization contains a vertical reflection, so \cref{thm:factor-word}(i) excludes factored translations. Hence the subfactorization representing $d$ consists entirely of ordinary reflections and gives $d\le_Tc$. Thus $d\in P_W$, contradicting
\[
  d\in P_F\setminus P_D,\qquad P_D=P_W\cap P_F.
\]
Therefore $u\in P_W\cap P_F=P_D$.

Write
\[
  \nu(u)=(r,\ldots,r).
\]
Since $du^{-1}\in G_C^+\cap G_F=G_F^+$ by \cref{lem:factor-cone}, right divisibility is coordinatewise in $G_F^+=\prod_iG_i^+$. Choose $j$ with $\nu_j(d_j)=0$. Then
\[
  0=\nu_j(d_j)=\nu_j\bigl((du^{-1})_j\bigr)+r.
\]
Both terms on the right are nonnegative, so $r=0$. Hence
\[
  \nu(u)=(0,\ldots,0).
\]
\end{proof}

\begin{proposition}[Rigidity of zero-winding coordinates]\label{prop:zero-rigidity}
Let $u=(u_1,\ldots,u_m)$ be a maximal $P_W$-right divisor of $b$, and suppose $d=d_C(b)\in P_F\setminus P_D$. Then
\[
  u_j=d_j\qquad(j\in Z(d)).
\]
\end{proposition}

\begin{proof}
Fix $j\in Z(d)$. By \cref{lem:location}, every coordinate of $u$ has winding zero. In the direct-product lattice $P_F$, put
\[
  u^{(j)}=(u_1,\ldots,u_{j-1},d_j,u_{j+1},\ldots,u_m).
\]
Then $u\preccurlyeq_Ru^{(j)}\preccurlyeq_Rd$, and $u^{(j)}$ has zero winding in every coordinate. By \cref{thm:factor-word}(iv), $u^{(j)}\in P_W\cap P_F=P_D$. By \cref{prop:divisor-detection} it is a $P_W$-right divisor of $b$. Maximality gives $u^{(j)}=u$.
\end{proof}

For $w\in P_c$, define
\begin{equation}\label{eq:reflection-subgroup}
  W(w)=\langle t\in T:t\le_Tw\rangle.
\end{equation}
By Dyer's reflection-subgroup theorem \cite[Theorem~3.3]{Dyer1990}, $W(w)$ carries a canonical Coxeter generating set and its canonical reflections are exactly $W(w)\cap T$. Paolini--Salvetti prove that $w$ is a Coxeter element of $W(w)$ and that an elliptic interval element gives a proper parabolic, hence finite, reflection subgroup \cite[Theorem~3.22 and Remark~3.23]{PaoliniSalvetti2021}.

\begin{definition}[Complete finite Coxeter component]\label{def:complete-component}
A nonidentity $e$ is an \emph{irreducible complete finite Coxeter component} of $w\in P_c$ if $W(e)$ is a finite irreducible connected component of Dyer's canonical Coxeter system of $W(w)$ and $e$ is the Coxeter element of that component in the component factorization of $w$.
\end{definition}

\begin{lemma}[Horizontal coordinate factors]\label{lem:horizontal-factors}
Let $u=(u_1,\ldots,u_m)\in P_D$ have winding zero. Then every $u_i$ is horizontal elliptic, $W(u_i)$ is finite, and
\[
  W(u)=W(u_1)\times\cdots\times W(u_m).
\]
If $W(u_i)=W_{i,1}\times\cdots\times W_{i,r_i}$ is its irreducible canonical decomposition and $u_i=e_{i,1}\cdots e_{i,r_i}$, every nonidentity $e_{i,a}$ is an irreducible complete finite Coxeter component of $u$.
\end{lemma}

\begin{proof}
By the McCammond--Sulway product decomposition,
\[
  P_F\cong P_1\times\cdots\times P_m,
\]
with coordinatewise order. For each $i$, put
\[
  \bar u_i=(1,\ldots,u_i,\ldots,1)\in P_F.
\]
Since $u$ has zero winding, a reduced factor word for $\bar u_i$ contains no factored translation; hence \cref{thm:factor-word}(iv) gives $\bar u_i\in P_W\cap P_F=P_D$. We identify $u_i$ with this ambient interval element. Its zero winding places it in the bottom row of the diagonal interval \cite[Remark~7.3]{McCammondSulway2017}, so it is horizontal elliptic \cite[Proposition~2.17]{PaoliniSalvetti2021}; consequently it is a Coxeter element of the finite reflection subgroup $W(u_i)$ \cite[Theorem~3.22 and Remark~3.23]{PaoliniSalvetti2021}.

The same bottom-row argument makes $u$ horizontal elliptic, and every reflection below $u$ is horizontal by \cite[Proposition~2.17]{PaoliniSalvetti2021}. Such a reflection belongs to a unique irreducible horizontal component. Since the order on $P_F$ is coordinatewise,
\[
  \{t\in T:t\le_Tu\}=\bigsqcup_{i=1}^m\{t\in T:t\le_Tu_i\}.
\]
Distinct horizontal components are orthogonal, so their reflection subgroups commute, and therefore
\[
  W(u)=W(u_1)\times\cdots\times W(u_m).
\]
Dyer's canonical simple reflections lie in $W(u)\cap T$. The displayed disjoint decomposition and commutativity across distinct horizontal factors therefore show that no edge of the canonical Coxeter graph joins two different $W(u_i)$. Hence Dyer's canonical Coxeter system splits into the canonical systems of the $W(u_i)$, and their irreducible canonical factors, together with their Coxeter elements, are precisely complete finite Coxeter components of $u$.
\end{proof}

\begin{theorem}[Common complete finite Coxeter component]\label{thm:common-core}
Let $W$ have affine rank at least three and let $1\ne b\in G_W^+$. If $b$ has no greatest $P_W$-right divisor, then all maximal $P_W$-right divisors of $b$ share a common nontrivial irreducible complete finite Coxeter component.
\end{theorem}

\begin{proof}
If $d_C(b)\in P_W$, \cref{prop:divisor-detection} makes it the greatest right divisor. Hence $d=d_C(b)\in P_F\setminus P_D$. Choose $j\in Z(d)$. This coordinate is nontrivial. Indeed, \cref{lem:diagonal-lifting} gives $x=fd\in G_D^+$ with completed gcd $d$. The common coordinate winding of $x$ is positive because $I(d)\ne\varnothing$; hence $x_j\ne1$. A right atom $\rho_j$ of $x_j$, embedded in the $j$th factor, right divides both $x$ and $c$, and therefore right divides $d$. Thus $d_j\ne1$.

By \cref{prop:zero-rigidity}, every maximal right divisor $u=(u_i)$ has $u_j=d_j$. The finite Coxeter group $W(d_j)$ and its irreducible canonical decomposition are independent of $u$. Any nonidentity irreducible factor is, by \cref{lem:horizontal-factors}, a complete finite component of every maximal divisor.
\end{proof}

\begin{lemma}[Interval quotients]\label{lem:interval-quotient}
Let $P=[1,c]_T$ be a Coxeter interval and let $u,v\in P$. If $u$ is a right divisor of $v$ in $G(P)^+$, then
\[
  q=vu^{-1}\in P,\qquad \ell_T(v)=\ell_T(q)+\ell_T(u).
\]
Hence
\[
  v=qu
\]
is a defining interval relation in $M(P)$.
\end{lemma}

\begin{proof}
Choose a positive $P$-word $Q$ representing $q=vu^{-1}$. Since degree extends to the interval group,
\[
  \deg Q=\ell_T(v)-\ell_T(u).
\]
Reflection-length subadditivity gives
\[
  \ell_T(v)\le \ell_T(q)+\ell_T(u)\le \deg Q+\ell_T(u)=\ell_T(v).
\]
Thus equality holds throughout. Hence
\[
  q\le_Tv\le_Tc,
\]
so $q\in P$ and $\ell_T(v)=\ell_T(q)+\ell_T(u)$. Therefore $v=qu$ is a defining interval relation.
\end{proof}

\begin{criterion}[Common right-divisor injectivity]\label{crit:injectivity}
Let $P=[1,c]_T$ be a finite-height Coxeter interval. Suppose that for every $1\ne b\in G(P)^+$ the maximal $P$-right divisors of $b$ have a common nonidentity $P$-right divisor. Then the natural map
\[
  M(P)\longrightarrow G(P)
\]
is injective.
\end{criterion}

\begin{proof}
Let $\pi:M(P)\twoheadrightarrow G(P)^+$ be the natural map. The common degree is well-defined by \cref{lem:interval-degree}. We induct on the common degree $N$ of two positive words $X,Y$ with $\pi(X)=\pi(Y)=b$. If $N=0$, both words are empty and there is nothing to prove. Assume $N>0$. Write $X=X_0u$ and $Y=Y_0v$, and extend $u,v$ to maximal right divisors $m_u,m_v$ of $b$. Let $e\ne1$ be a common right divisor. By \cref{lem:interval-quotient}, there are interval elements with defining relations
\[
  m_u=a_uu=h_ue,\qquad m_v=a_vv=h_ve.
\]
Since $m_u\preccurlyeq_Rb$, write $b=q_um_u$ in $G(P)^+$ and choose a positive word $Q_u$ for $q_u$. Then
\[
  \pi(X_0)=bu^{-1}=q_ua_u=\pi(Q_ua_u).
\]
Both sides have degree $N-\ell_T(u)<N$, so induction gives $X_0=Q_ua_u$ in $M(P)$. Hence, using only interval relations,
\[
  X=Q_ua_uu=Q_um_u=Q_uh_ue=:R_ue.
\]
Likewise $Y=R_ve$. The prefixes $R_u,R_v$ have common group image $be^{-1}$ and degree $N-\ell_T(e)<N$, so a second induction gives $R_u=R_v$ and hence $X=Y$.
\end{proof}

\begin{theorem}[Affine interval-monoid embedding]\label{thm:interval-embedding}
For every finite-rank irreducible crystallographic affine Coxeter system and every Coxeter element $c$, the natural map
\[
  M_c\hookrightarrow G_W
\]
is injective. Consequently $M_c$ is left- and right-cancellative. In rank at least three,
\[
  M_c\cong G_W^+=G_C^+\cap G_W.
\]
\end{theorem}

\begin{proof}
In rank two this is \cref{rem:rank-two}. Assume rank at least three. Let $1\ne b\in G_W^+$. If it has a greatest $P_W$-right divisor, that divisor is common to all maximal ones. Otherwise \cref{thm:common-core} supplies a common nonidentity complete finite component, which is itself a $P_W$-right divisor. Since $P_W$ has height at most $\ell_T(c)$, \cref{crit:injectivity} applies. The positive-image identification is \cref{thm:positive-saturation}, and cancellation follows from the group embedding.
\end{proof}

\section{Hereditary support realizations and half-orbit rectification}\label{sec:rectification}

We now construct the exceptional-cluster model in which principal fibres become intrinsic. Throughout this section, $\HH$ denotes a basic connected tame hereditary module category realizing an affine Cartan datum, with Coxeter element determined by an exceptional ordering of its simple objects. For the main construction we use the realization fixed in \cref{sec:coxeter-hereditary}; the notation $D,\tau,\tauD$ is used for the corresponding duality and Auslander--Reiten translations.

Let $P_i$ and $I_i$ be the indecomposable projective and injective objects corresponding to $\alpha_i$, and set
\begin{equation}\label{eq:pi-iota}
  \pi_i=\dim P_i=s_n\cdots s_{i+1}(\alpha_i),\qquad
  \iota_i=\dim I_i=s_1\cdots s_{i-1}(\alpha_i).
\end{equation}
Then
\begin{equation}\label{eq:cpi}
  c\pi_i=-\iota_i,\qquad t_{\iota_i}=ct_{\pi_i}c^{-1}.
\end{equation}

\subsection{Real vertices and support representatives}
Reading--Stella classify the $\tau_c$-orbits \cite[Proposition~3.12]{ReadingStella2020}. If $\beta_{i,m}=\tau_c^m(-\alpha_i)$, then
\begin{equation}\label{eq:RS-orbit}
  \beta_{i,m}=\begin{cases}
    c^{m+1}\pi_i,&m\le-1,\\
    -\alpha_i,&m=0,\\
    c^{m-1}\iota_i,&m\ge1.
  \end{cases}
\end{equation}
All other real vertices lie in finitely many finite $\tau_c$-orbits and are positive regular roots.

\begin{proposition}[Positive vertex identity]\label{prop:positive-vertices}
One has
\[
  \Phi_c^{\re}\cap\Phi_+=\{\alpha\in\Phi_+ : t_\alpha\le_Tc\}.
\]
Consequently positive Reading--Stella vertices are in bijection with indecomposable exceptional modules via their dimension vectors.
\end{proposition}

\begin{proof}
Let $U^c$ be the Reading--Stella horizontal hyperplane and $\Upsilon^c=\Phi\cap U^c$. Hanson--Reading decompose the reflections below $c$ into all vertical reflections and the horizontal reflections below $c$ \cite[Proposition~6.1]{HansonReading2025}. The positive roots of vertical reflections are exactly $\Phi_+\setminus U^c$. If $\Xi^c$ is the canonical simple system of $\Upsilon^c$ and $\beta\in\Xi^c$ has $c$-cycle length $r_\beta$, the formula following \cite[Proposition~6.1]{HansonReading2025} gives the horizontal positive roots
\[
  \beta^{(p)}=\beta+c\beta+\cdots+c^{p-1}\beta,\qquad1\le p<r_\beta.
\]
These are precisely the roots in the finite regular part $\Lambda_c^{\re}$ of Reading--Stella's definition \cite[Definition~3.1 and Lemma~3.10]{ReadingStella2020}. The final assertion is the rank-one case of \cref{thm:HR}.
\end{proof}

\subsection{Hereditary support realization}
Let $\mathscr S(\HH)$ be the hereditary support complex. Its vertices are the indecomposable exceptional modules together with formal symbols $P_i[1]$; a finite set represented by
\[
  M\oplus\bigoplus_{i\in J}P_i[1]
\]
is a face when $M$ is basic rigid and $\Hom(P_i,M)=0$ for all $i\in J$. The complex is flag.

Let
\[
  \CC_\HH=D^b(\HH)/(\tauD^{-1}[1])
\]
be the hereditary cluster category. The orbit-category calculation of Buan--Marsh--Reineke--Reiten--Todorov gives, for modules $X,Y$ and projective $P$ \cite[Section~1, especially Propositions~1.5--1.7]{BMRRT2006},
\begin{align}
  \Ext^1_{\CC_\HH}(X,Y)&\cong\Ext^1_\HH(X,Y)\oplus D\Ext^1_\HH(Y,X),\label{eq:cluster-ext}\\
  \Hom_{\CC_\HH}(P,Y)&\cong\Hom_\HH(P,Y).\label{eq:cluster-hom}
\end{align}
Heredity leaves only the zeroth and first orbit summands, and derived Serre duality identifies the latter.

\begin{lemma}[Support compatibility as cluster rigidity]\label{lem:support-rigid}
For distinct vertices $U,V$ of $\mathscr S(\HH)$,
\[
  \{U,V\}\in\mathscr S(\HH)\iff \Ext^1_{\CC_\HH}(U,V)=0.
\]
Hence support compatibility is preserved by triangle autoequivalences of $\CC_\HH$.
\end{lemma}

\begin{proof}
For module vertices this is \cref{eq:cluster-ext}. If $U=P_i[1]$ and $V=X$ is a module, then
\[
  \Ext^1_{\CC_\HH}(P_i[1],X)\cong\Hom_{\CC_\HH}(P_i,X)\cong\Hom_\HH(P_i,X),
\]
which is the support condition. Two shifted projectives are compatible because the corresponding first extension group between projectives vanishes.
\end{proof}

The autoequivalence induced by $\tauD$ permutes the standard fundamental domain $\ind(\HH)\sqcup\{P_i[1]\}$. Let $\widehat\tau$ denote the induced permutation of representatives:
\begin{equation}\label{eq:support-rotation}
  \widehat\tau(U)=\begin{cases}
    \tau U,&U\text{ nonprojective module},\\
    P_i[1],&U=P_i,\\
    I_i,&U=P_i[1].
  \end{cases}
\end{equation}
Define the root label
\begin{equation}\label{eq:root-label}
  \ell(U)=\begin{cases}
    \dim U,&U\in\ind\Exc(\HH),\\
    -\alpha_i,&U=P_i[1].
  \end{cases}
\end{equation}

\begin{lemma}[Rotation equivariance]\label{lem:rotation-equiv}
The permutation $\widehat\tau$ preserves support compatibility and
\[
  \ell(\widehat\tau U)=\tau_c(\ell(U)).
\]
\end{lemma}

\begin{proof}
Compatibility preservation is \cref{lem:support-rigid}. For a nonprojective module, $\dim\tau U=c(\dim U)$ with the present Coxeter convention \cite[Section~2.2]{HansonReading2025}, while $\tau_c$ agrees with the linear action of $c$ away from the splice \cite[Proposition~3.12]{ReadingStella2020}. At the splice, \cref{eq:pi-iota,eq:support-rotation} give
\[
  \ell(\widehat\tau P_i)=-\alpha_i=\tau_c(\pi_i),\qquad
  \ell(\widehat\tau P_i[1])=\iota_i=\tau_c(-\alpha_i).
\]
\end{proof}

Let $\Upsilon^c=\Phi\cap U^c$ and let $\Xi^c$ be its canonical simple system. For $\beta\in\Xi^c$ of cycle length $r_\beta$, write
\[
  \beta^{(p)}=\beta+c\beta+\cdots+c^{p-1}\beta,\qquad1\le p<r_\beta.
\]
The exceptional objects in nonhomogeneous tubes have the standard Euclidean description used in \cite[Remark~8.1 and Section~8]{HansonReading2025}. Let $R_{\beta,p}$ denote the exceptional indecomposable of quasi-length $p$ with $\dim R_{\beta,p}=\beta^{(p)}$.

For such a root put $I(\beta,p)=\{\beta,c\beta,\ldots,c^{p-1}\beta\}$, a proper cyclic interval in the corresponding horizontal component.

\begin{lemma}[Tube interval criterion]\label{lem:tube}
For distinct exceptional indecomposables $R_{\beta,p}$ and $R_{\gamma,q}$ in the same nonhomogeneous tube,
\[
  \Ext^1(R_{\beta,p},R_{\gamma,q})=0=\Ext^1(R_{\gamma,q},R_{\beta,p})
\]
if and only if $I(\beta,p)$ and $I(\gamma,q)$ are nested or spaced.
\end{lemma}

\begin{proof}
Index quasi-simples by $\mathbb Z/r\mathbb Z$ so that $\tau\xi_i=\xi_{i+1}$. On the universal cyclic cover, let $R[a,b]$ have quasi-composition factors $\xi_a,\ldots,\xi_b$, with $0\le b-a<r-1$. The standard tube formula is
\[
  \Hom(R[a,b],R[u,v])\ne0\iff u+kr\le a\le v+kr\le b
\]
for some $k\in\mathbb Z$ \cite[Corollary~X.2.7(a)]{SimsonSkowronski2007}. Auslander--Reiten duality can be read with ordinary Hom here. Indeed, the AR formula quotients $\Hom(R[u,v],\tau R[a,b])$ by maps factoring through injectives. If a factorization $R[u,v]\to I\to\tau R[a,b]$ passes through an injective $I$, its second arrow is zero by tame separation $\Hom(I,R)=0$ for regular $R$ \cite[Sections~2.3--3.1]{Ringel1984}. Hence the quotient agrees with ordinary Hom. Since $\tau R[a,b]=R[a+1,b+1]$,
\[
  \Ext^1(R[a,b],R[u,v])\ne0\iff a+1\le u+kr\le b+1\le v+kr
\]
for some $k$. Interchanging the intervals gives the reverse extension group. These nonvanishing configurations are exactly crossing proper cyclic intervals and cyclically adjacent disjoint intervals. Their complement is nested-or-spaced.
\end{proof}

\begin{proposition}[Finite-orbit regular compatibility]\label{prop:regular-compatibility}
Let $\beta,\gamma\in\Lambda_c^{\re}$ be distinct finite-orbit vertices, and let $M_\beta,M_\gamma$ be their exceptional regular modules. Then
\[
  \beta\sim_c\gamma
  \quad\Longleftrightarrow\quad
  \Ext^1(M_\beta,M_\gamma)=0=\Ext^1(M_\gamma,M_\beta).
\]
The same statement holds for any basic connected tame hereditary realization of an affine Cartan datum.
\end{proposition}

\begin{proof}
If $M_\beta$ and $M_\gamma$ lie in the same nonhomogeneous tube, \cref{lem:tube} identifies mutual first-extension vanishing with the nested-or-spaced condition on the associated proper cyclic intervals. Reading--Stella identify exactly this condition with $c$-compatibility \cite[Definition~5.11 and Proposition~5.12]{ReadingStella2020}.

If the modules lie in distinct nonhomogeneous tubes, Hanson--Reading's tame description gives
\[
  \Hom(M_\beta,M_\gamma)=0=\Hom(M_\gamma,M_\beta),
  \qquad
  \Ext^1(M_\beta,M_\gamma)=0=\Ext^1(M_\gamma,M_\beta)
\]
\cite[Lemma~8.3]{HansonReading2025}. Their cyclic supports lie in distinct components of the horizontal canonical simple system, hence are spaced in the sense of Reading--Stella and are compatible by \cite[Definition~5.11 and Proposition~5.12]{ReadingStella2020}. The realization-independence of the exceptional and nonhomogeneous-tube description needed here is recorded in \cite[Remark~8.1]{HansonReading2025}.
\end{proof}

\begin{lemma}[Negative-simple compatibility]\label{lem:negative-simple}
For a positive real vertex $\beta$,
\[
  -\alpha_i\sim_c\beta
  \iff \Hom(P_i,M_\beta)=0
  \iff \{P_i[1],M_\beta\}\in\mathscr S(\HH).
\]
Distinct negative simple roots are compatible, just as the corresponding shifted projectives are support-compatible.
\end{lemma}

\begin{proof}
Reading--Stella's initial condition gives $( -\alpha_i\Vert\beta)_c=[\beta:\alpha_i]$ \cite[Equation~(4.3)]{ReadingStella2020}. The zero locus of the compatibility degree is symmetric because swapping the variables changes the degree only by a positive scalar \cite[Proposition~4.12 and the discussion after Definition~5.1]{ReadingStella2020}. The $i$th simple-root coordinate of $\dim M_\beta$ vanishes exactly when $\Hom(P_i,M_\beta)=0$. The final assertions are the corresponding initial-cluster statements.
\end{proof}

\begin{proposition}[Hereditary support realization]\label{prop:support-realization}
The same-root labelling \cref{eq:root-label} induces a simplicial isomorphism
\[
  \ell:\mathscr S(\HH)\xrightarrow{\sim}\DeltaReal.
\]
More generally, the same statement holds for any basic connected tame hereditary realization of an affine root datum, with Coxeter element determined by an exceptional ordering of the simple objects.
\end{proposition}

\begin{proof}
The map $\ell$ is a vertex bijection by \cref{prop:positive-vertices}. If at least one label belongs to an infinite $\tau_c$-orbit, a simultaneous power of $\tau_c$ sends it to a negative simple root \cite[Proposition~3.12(4)]{ReadingStella2020}. Reading--Stella compatibility is $\tau_c$-invariant \cite[Equation~(4.8) and Proposition~5.4(1)]{ReadingStella2020}; \cref{lem:rotation-equiv} makes the same rotation on support vertices, so the comparison reduces to \cref{lem:negative-simple}.

If both labels lie in finite orbits, they are regular roots in $\Lambda_c^{\re}$, and \cref{prop:regular-compatibility} gives the required equivalence. Thus $\ell$ preserves and reflects every edge, and both complexes are flag. The general realization statement follows from the same tame-hereditary inputs, together with \cite[Remark~8.1]{HansonReading2025} for realization-independence of the exceptional regular part.
\end{proof}

\subsection{Half-orbit rectification}

\begin{definition}[Half-orbit relabelling]\label{def:eta}
On an infinite orbit define
\[
  \eta_c(\beta_{i,m})=\begin{cases}
    \beta_{i,m},&m<0,\\
    \beta_{i,m+1},&m\ge0,
  \end{cases}
\]
and let $\eta_c$ be the identity on finite regular orbits. For a real vertex $\beta$, let $X_\beta$ be the exceptional module with $\dim X_\beta=\eta_c(\beta)$ and put $\lambda_c(\beta)=t_{\eta_c(\beta)}$.
\end{definition}

On each infinite orbit, \cref{def:eta} removes the negative-simple splice in \cref{eq:RS-orbit} and maps bijectively onto the two positive tails, while finite regular orbits are fixed. Hence $\eta_c$ is a bijection from $V(\DeltaReal)$ onto $\Phi_c^{\re}\cap\Phi_+$, the positive real Schur-root set by \cref{prop:positive-vertices}, and therefore
\[
  \lambda_c:V(\DeltaReal)\xrightarrow{\sim}\{t\in T:t\le_Tc\}.
\]
Using \cref{eq:cpi}, the reflection labels are
\[
  \lambda_c(\beta)=\begin{cases}
    t_\beta,&\beta\text{ preprojective or regular},\\
    ct_\beta c^{-1},&\beta\text{ positive preinjective},\\
    ct_{\pi_i}c^{-1}=t_{\iota_i},&\beta=-\alpha_i.
  \end{cases}
\]

\begin{remark}
The map $\eta_c$ is only a relabelling of vertices. It is not a functor, a categorical autoequivalence or a mutation operation, and no simpliciality statement about $\eta_c$ itself is used.
\end{remark}

\subsection{Defect and the rectified exceptional cluster complex}
Let $\delta$ be the primitive positive null root. Define
\begin{equation}\label{eq:defect}
  \partial(X)=E_{c^{-1}}(\delta,\dim X).
\end{equation}
Then
\begin{equation}\label{eq:tame-trichotomy}
\begin{aligned}
  \partial(X)<0&\iff X\text{ is preprojective},\\
  \partial(X)=0&\iff X\text{ is regular},\\
  \partial(X)>0&\iff X\text{ is preinjective}.
\end{aligned}
\end{equation}
This normalization agrees with \cite[Section~8.1]{HansonReading2025}.

Define the support section
\begin{equation}\label{eq:support-section}
  \rho_c(X)=\begin{cases}
    X,&\partial(X)\le0,\\
    \tau^{-1}X,&\partial(X)>0\text{ and }X\text{ is not injective},\\
    P_i[1],&X=I_i.
  \end{cases}
\end{equation}

\begin{lemma}[Support-section identity]\label{lem:support-section}
For every indecomposable exceptional module $X$,
\[
  \ell(\rho_c(X))=\eta_c^{-1}(\dim X).
\]
Thus $\rho_c$ is a bijection from ordinary exceptional modules to support vertices.
\end{lemma}

\begin{proof}
On finite regular orbits and the preprojective half of an infinite orbit, $\eta_c$ is the identity. On the positive preinjective half, $\eta_c^{-1}$ moves one step backwards, realized by $\tau^{-1}$. At the injective boundary, $\eta_c(-\alpha_i)=\iota_i=\dim I_i$ and $-\alpha_i$ is represented by $P_i[1]$.
\end{proof}

\begin{lemma}[Defect-dependent compatibility]\label{lem:defect-compatibility}
Let $X,Y$ be indecomposable exceptional modules.
\begin{enumerate}[label=\textup{(\roman*)}]
\item If $\partial(X),\partial(Y)\le0$, or if both are positive, then
\[
  \{\rho_c(X),\rho_c(Y)\}\in\mathscr S(\HH)
  \iff \Ext^1(X,Y)=0=\Ext^1(Y,X).
\]
\item If $\partial(N)\le0<\partial(I)$, then
\[
  \{\rho_c(N),\rho_c(I)\}\in\mathscr S(\HH)
  \iff \Hom(N,I)=0.
\]
\end{enumerate}
\end{lemma}

\begin{proof}
If both defects are nonpositive, the support representatives are the modules themselves. If both are positive, $\rho_c(X)$ and $\rho_c(Y)$ represent $\tau_{\CC_\HH}^{-1}X$ and $\tau_{\CC_\HH}^{-1}Y$. By \cref{lem:support-rigid} and invariance under the triangle autoequivalence, their compatibility is equivalent to $\Ext^1_{\CC_\HH}(X,Y)=0$, which by \cref{eq:cluster-ext} is mutual ordinary $\Ext^1$-vanishing. This includes injective boundary objects.

Suppose $\partial(N)\le0<\partial(I)$. If $I$ is not injective, put $J=\tau^{-1}I$, so $J$ is preinjective and $\rho_c(I)=J$. Tame separation gives $\Ext^1(N,J)=0$. Derived Serre duality gives
\[
  D\Hom(N,I)\cong \Hom_{D^b(\HH)}(I,\tauD N[1])\cong\Ext^1(J,N).
\]
Thus the two module representatives are support-compatible exactly when $\Hom(N,I)=0$. If $I=I_i$, support compatibility with $P_i[1]$ is $\Hom(P_i,N)=0$, which is equivalent to $D\Hom(N,I_i)=0$ by \cref{eq:proj-inj-duality}.
\end{proof}

\subsection{Rectified exceptional reconstruction}

\begin{definition}[Rectified exceptional $c$-cluster complex]\label{def:rectified}
The rectified exceptional $c$-cluster complex $\Delta_c^{\mathrm{exc}}(\HH)$ is the flag complex on $\ind\Exc(\HH)$ whose edges are prescribed by \cref{lem:defect-compatibility}: same-side pairs are adjacent exactly when both $\Ext^1$ groups vanish, while a mixed pair $\partial(N)\le0<\partial(I)$ is adjacent exactly when $\Hom(N,I)=0$.
\end{definition}

In the algebraic and fibre-theoretic sections put $K_c:=\Delta_c^{\mathrm{exc}}(\HH)$.

\begin{theorem}[Rectified exceptional reconstruction]\label{thm:rectified}
The vertex map
\[
  X\longmapsto\eta_c^{-1}(\dim X)
\]
induces a simplicial isomorphism
\[
  \Delta_c^{\mathrm{exc}}(\HH)\xrightarrow{\sim}\Delta_c^{\re}(\Phi).
\]
More generally, for every basic connected tame hereditary realization $\mathcal A$ with affine root system $\Psi$ and Coxeter element $e$ determined by an exceptional ordering of the simple objects,
\[
  \Delta_e^{\mathrm{exc}}(\mathcal A)\xrightarrow{\sim}\Delta_e^{\re}(\Psi).
\]
\end{theorem}

\begin{proof}
By \cref{lem:defect-compatibility}, the support section $\rho_c$ preserves and reflects edges; it is a vertex bijection by \cref{lem:support-section}, and both complexes are flag. Hence it is simplicial. \Cref{prop:support-realization} identifies the support complex with $\Delta_c^{\re}(\Phi)$, and \cref{lem:support-section} gives $\ell(\rho_c(X))=\eta_c^{-1}(\dim X)$ for every exceptional module $X$. The proofs of \cref{prop:support-realization,lem:support-section,lem:defect-compatibility} were formulated for a basic connected tame hereditary realization; applying those statements to $\mathcal A$ gives the general assertion.
\end{proof}

\begin{remark}[Realization invariance]
If two tame hereditary realizations have isomorphic oriented generalized Cartan lattices carrying ordered simple basis to ordered simple basis, \cref{thm:rectified} identifies both rectified exceptional complexes with the same Reading--Stella complex. Thus realization invariance is a consequence, not an input, of the direct reconstruction.
\end{remark}

\section{Exceptional wide subcategories, intrinsic face weights, and principal fibres}\label{sec:weights}

\subsection{Exceptional ordering of rectified faces}
We use tame separation
\begin{equation}\label{eq:tame-separation}
  \Hom(I,R)=0=\Hom(I,P)
\end{equation}
for preinjective $I$, regular $R$, and preprojective $P$; see \cite[Sections~2.3--3.1]{Ringel1984}. Auslander--Reiten duality gives
\begin{equation}\label{eq:tame-ext}
  \Ext^1(P,I)=0=\Ext^1(R,I).
\end{equation}

\begin{lemma}[Exceptional ordering of faces]\label{lem:face-exceptional-order}
Every face of the rectified exceptional complex can be ordered into an exceptional sequence. More precisely, positive-defect vertices may be placed before nonpositive-defect vertices.
\end{lemma}

\begin{proof}
Write $F=F_+\sqcup F_{\le0}$. Within each part, rectified compatibility gives mutual $\Ext^1$-orthogonality. Thus the direct sums are partial tilting objects, and the indecomposable summands in each part admit exceptional orderings by \cite[Lemma~2.2(a)(i)]{BuanReitenThomas2011}. Concatenate an exceptional ordering of $F_+$ before one of $F_{\le0}$. For $I\in F_+$ and $N\in F_{\le0}$, mixed compatibility gives $\Hom(N,I)=0$, while \cref{eq:tame-ext} gives $\Ext^1(N,I)=0$, exactly the cross-block conditions for an exceptional sequence.
\end{proof}

\subsection{Intrinsic face weights and principal fibres}

\begin{definition}[Intrinsic exceptional face weight]\label{def:weight}
For $F\in K_c$, define
\begin{equation}\label{eq:face-weight}
  \WW_F=\wide\langle X:X\in F\rangle,\qquad
  w(F)=\cox(\WW_F)\in P_c,\qquad
  \omega(F)=w(F)\in M_c.
\end{equation}
\end{definition}
The element $w(F)$ is independent of the exceptional ordering by \cref{thm:HR}.

\begin{theorem}[Rectified face map]\label{thm:face-map}
For faces $G\subseteq F$ of $K_c$,
\[
  w(\varnothing)=1,\qquad \ell_T(w(F))=|F|,\qquad w(G)\le_Tw(F).
\]
For every facet $C$ of $K_c$, $\WW_C=\HH$ and $w(C)=c$.
\end{theorem}

\begin{proof}
Order $F$ into an exceptional sequence. It is complete inside its wide closure, whose rank is therefore $|F|$. Rank preservation in \cref{thm:HR} gives $\ell_T(w(F))=|F|$. If $G\subseteq F$, then $\WW_G\subseteq\WW_F$, so $w(G)\le_Tw(F)$. If $C$ is a facet, \cref{thm:rectified} and \cite[Proposition~5.14]{ReadingStella2020} give $|C|=n$; \cref{lem:face-exceptional-order} then gives a complete exceptional sequence of length $n$, whence $\WW_C=\HH$ and $w(C)=c$.
\end{proof}

\begin{lemma}[Absolute order and right divisibility]\label{lem:absolute-right-div}
For $u,v\in P_c$,
\[
  u\le_Tv\iff u\preccurlyeq_Rv\text{ in }M_c.
\]
\end{lemma}

\begin{proof}
If $u\le_Tv$, put $x=vu^{-1}$. By \cref{eq:absolute-right}, $\ell_T(v)=\ell_T(x)+\ell_T(u)$; since $v\le_Tc$, this reduced factorization extends to one of $c$, so $x\in P_c$ and $v=xu$ is an interval relation.

Conversely, suppose $v=au$ in $M_c$ and let $\pi:M_c\to W$ be the product map. A monoid generator of degree $d$ has a factorization into $d$ reflections, hence $\ell_T(\pi(a))\le\deg a=\ell_T(v)-\ell_T(u)$. Since $\pi(a)=vu^{-1}$, the triangle inequality forces equality throughout, and \cref{eq:absolute-right} gives $u\le_Tv$.
\end{proof}

\begin{corollary}\label{cor:weight-axioms}
The map $\omega$ satisfies the face-weight axioms used in \cref{sec:weighted}.
\end{corollary}

For $w\in P_c$, put
\begin{equation}\label{eq:principal-vertices}
  \WW_w=\cox^{-1}(w),\qquad \mathcal A_c(w)=\ind\Exc(\WW_w)\subseteq V(K_c),
\end{equation}
and define the principal fibre
\begin{equation}\label{eq:principal-fibre}
  K_w=\{F\in K_c:\omega(F)\preccurlyeq_Rw\}.
\end{equation}
Here $\WW_w$ is a categorical wide subcategory and must not be confused with the reflection subgroup $W(w)$.

\begin{proposition}[Principal fibres are induced]\label{prop:principal-induced}
For $w\in P_c$,
\[
  K_w=K_c[\mathcal A_c(w)].
\]
\end{proposition}

\begin{proof}
For a face $F$, \cref{lem:absolute-right-div,thm:HR} give
\[
  \omega(F)\preccurlyeq_Rw\iff w(F)\le_Tw\iff \WW_F\subseteq\WW_w.
\]
The last condition is exactly that every vertex of $F$ lies in $\WW_w$.
\end{proof}

\subsection{Categorical decomposition of principal fibres}\label{sec:decomposition}

Let $1\ne w\in P_c$. Decompose
\begin{equation}\label{eq:block-decomp}
  \WW_w=\VV_1\oplus\cdots\oplus\VV_s
\end{equation}
into connected exact blocks. For $i\ne j$,
\begin{equation}\label{eq:block-orth}
  \Hom(\VV_i,\VV_j)=0=\Ext^1(\VV_i,\VV_j)
\end{equation}
in both directions. Each block is a connected hereditary module category by \cref{thm:HR}. Let $B$ be the ambient symmetrized Euler form. Because $B$ is positive semidefinite with one-dimensional radical $\mathbb R\delta$, the radical of its restriction to the real span $L_i$ of $K_0(\VV_i)$ is
\begin{equation}\label{eq:block-radical}
  \rad(B|_{L_i})=L_i\cap\mathbb R\delta.
\end{equation}
Indeed, for a positive-semidefinite form, $B(x,x)=0$ already forces $x\in\rad B$. Thus, by the standard finite/affine classification of connected symmetrizable Cartan data \cite[Sections~2.2 and 8]{HansonReading2025}, the datum of $\VV_i$ is either positive definite (Dynkin) or positive semidefinite of corank one (affine). Moreover, two different affine blocks cannot occur: each affine block would contain a nonzero null vector proportional to $\delta$, whereas the direct-sum decomposition of $K_0(\WW_w)$ makes the subspaces $L_i$ intersect trivially. Hence at most one block is affine.

For a connected block $\VV\subseteq\WW_w$, define
\begin{equation}\label{eq:block-complex}
  K_\VV=K_c[\ind\Exc(\VV)].
\end{equation}

\begin{theorem}[Embedded component join]\label{thm:embedded-join}
For \cref{eq:block-decomp},
\[
  K_w=K_{\VV_1}*\cdots*K_{\VV_s}
\]
as an equality of face sets in the ambient complex $K_c$.
\end{theorem}

\begin{proof}
Objects in distinct blocks are orthogonal for Hom and Ext by \cref{eq:block-orth}. The compatibility criterion \cref{lem:defect-compatibility} therefore makes every cross-block pair compatible, including pairs on opposite defect sides. Since $K_c$ is flag, the union of arbitrary componentwise faces is an ambient face. Conversely, \cref{prop:principal-induced} says that the vertex set of $K_w$ is the disjoint union of the component vertex sets.
\end{proof}

\begin{lemma}[Cartan datum of a connected block]\label{lem:block-cartan}
Let $\VV$ be a connected block of an exceptional wide subcategory of $\HH$. The exact wide inclusion $\iota:\VV\hookrightarrow\HH$ induces a morphism of generalized Cartan lattices
\[
  \iota_*:K_0(\VV)\longrightarrow K_0(\HH),
\]
and embeddings of the associated real root system and Weyl group
\[
  \Phi_\VV^{\re}\hookrightarrow\Phi^{\re},\qquad W_\VV\hookrightarrow W,
  \qquad s_\alpha\mapsto s_{\iota_*(\alpha)}.
\]
Under this embedding the local Coxeter element maps to $e_\VV=\cox(\VV)\in[1,c]_T$.
\end{lemma}

\begin{proof}
The inclusion is fully faithful, exact and extension-closed. Hubery--Krause therefore give the generalized Cartan lattice morphism and the corresponding injections on roots and Weyl groups \cite[Theorem~1.1, Theorems~5.2, 5.6 and Corollary~7.3]{HuberyKrause2016}. If $L_1,\ldots,L_r$ are the local simples in an exceptional ordering, they form a complete exceptional sequence in $\VV$. Thus the local Coxeter element $s_{[L_1]}\cdots s_{[L_r]}\in W_\VV$ maps under the displayed Weyl-group embedding to $t_{\dim L_1}\cdots t_{\dim L_r}=\cox(\VV)$ by \cref{thm:HR}.
\end{proof}

\section{Dynkin and affine blocks, and complete Coxeter components}\label{sec:factors}

\subsection{Dynkin blocks and a split HRS tilt}
Fix a connected Dynkin block $\VV$. Divide its indecomposables by ambient defect:
\begin{equation}\label{eq:TF}
  \TT=\add\{I\in\ind\VV:\partial(I)>0\},\qquad
  \FF=\add\{N\in\ind\VV:\partial(N)\le0\}.
\end{equation}

\begin{lemma}[Split torsion pair]\label{lem:split-torsion}
The pair $(\TT,\FF)$ is a split torsion pair in $\VV$, and
\[
  \Ext^1_\VV(\FF,\TT)=0.
\]
\end{lemma}

\begin{proof}
The category $\VV$ is representation-finite, so each object is a direct sum of indecomposables from $\TT$ and $\FF$. Tame separation gives $\Hom(\TT,\FF)=0$, hence these decompositions provide split torsion-pair sequences. Let $F\in\FF$ and $T\in\TT$. If $F$ is projective, $\Ext^1(F,T)=0$. Otherwise AR duality identifies $D\Ext^1_\HH(F,T)$ with a quotient of $\Hom_\HH(T,\tau F)$; whenever nonzero, $\tau F$ is non-preinjective, whereas $T$ is preinjective, so tame separation gives zero. The wide embedding computes the same first extension group.
\end{proof}

Form the backward HRS heart
\begin{equation}\label{eq:HRS-heart}
  \VV^\dagger=\{E\in D^b(\VV):H^0(E)\in\FF,\ H^1(E)\in\TT,\ H^j(E)=0\ (j\ne0,1)\}.
\end{equation}
For a split torsion pair the realization functor $D^b(\VV^\dagger)\to D^b(\VV)$ is an equivalence \cite[Example~4.2(2)]{ChenHanZhou2019}; compare \cite[Chapter~I, Section~2]{HRS1996}.

\begin{lemma}[The tilted heart is hereditary]\label{lem:tilted-hereditary}
The heart $\VV^\dagger$ is hereditary. Its indecomposable objects are
\[
  \ind\VV^\dagger=\ind\FF\sqcup\{T[-1]:T\in\ind\TT\},
\]
and for $F,F_i\in\FF$, $T,T_i\in\TT$,
\begin{align*}
  \Ext^1_{\VV^\dagger}(F_1,F_2)&=\Ext^1_\VV(F_1,F_2),\\
  \Ext^1_{\VV^\dagger}(T_1[-1],T_2[-1])&=\Ext^1_\VV(T_1,T_2),\\
  \Ext^1_{\VV^\dagger}(F,T[-1])&=\Hom_\VV(F,T),\\
  \Ext^1_{\VV^\dagger}(T[-1],F)&=0.
\end{align*}
\end{lemma}

\begin{proof}
Full faithfulness of the realization functor allows all extensions to be computed in $D^b(\VV)$. The four identities follow from shifts and heredity. The only potentially nonzero second extension is
\[
  \Ext^2_{\VV^\dagger}(F,T[-1])=\Ext^1_\VV(F,T),
\]
which vanishes by \cref{lem:split-torsion}. Higher extensions vanish for degree reasons.

Every object of the HRS heart fits into its canonical short exact sequence
\[
  0\longrightarrow F\longrightarrow E\longrightarrow T[-1]\longrightarrow0.
\]
The last displayed formula gives $\Ext^1_{\VV^\dagger}(T[-1],F)=0$, so this sequence splits. Hence every indecomposable belongs to exactly one of the two displayed classes.
\end{proof}

\subsection{Length and module realization of the tilted heart}

\begin{lemma}[Length and module realization]\label{lem:tilted-length}
The heart $\VV^\dagger$ is a length category. Moreover, there is a finite-dimensional representation-finite hereditary $K$-algebra $\Gamma_\dagger$ and an exact equivalence
\[
  \VV^\dagger\simeq\operatorname{mod}\Gamma_\dagger.
\]
Its Euler lattice is isometric to that of $\VV$; in particular, it has the same finite Dynkin Coxeter type as $\VV$.
\end{lemma}

\begin{proof}
Since $\VV$ is Dynkin, it has only finitely many indecomposable objects. The heart $\VV^\dagger$ is Hom-finite and Krull--Schmidt, and \cref{lem:tilted-hereditary} shows that it has only finitely many indecomposable isomorphism classes. Hence $\VV^\dagger$ is a length category by \cite[Theorem~7]{GrimelandJacobsen2015}.

By \cref{thm:HR}, write $\VV\simeq\operatorname{mod}\Lambda_\VV$ for a finite-dimensional hereditary $K$-algebra. The category $\VV^\dagger$ is the heart of a bounded $t$-structure on $D^b(\operatorname{mod}\Lambda_\VV)$. K\"onig--Yang's length-heart theorem therefore gives an exact equivalence
\[
  \VV^\dagger\simeq\operatorname{mod}\Gamma_\dagger
\]
for a finite-dimensional algebra $\Gamma_\dagger$ \cite[Corollary~6.2]{KoenigYang2014}. The algebra is hereditary by \cref{lem:tilted-hereditary} and representation-finite because $\VV^\dagger$ has only finitely many indecomposables. Finally, the realization equivalence $D^b(\VV^\dagger)\simeq D^b(\VV)$ preserves the Euler pairing on $K_0$. Thus the Euler lattices are isometric, and in particular the two categories have the same finite Dynkin Coxeter type.
\end{proof}

\subsection{The positive finite cluster ball}

\begin{lemma}[Dynkin rigid complex and positive clusters]\label{lem:Dynkin-positive}
Let $\mathscr R(\VV^\dagger)$ be the flag complex of pairwise $\Ext^1$-orthogonal indecomposable exceptional objects of the hereditary Dynkin heart $\VV^\dagger$. If $c_\dagger$ is the Coxeter element determined by an exceptional ordering of its simple objects, then the dimension-vector map identifies $\mathscr R(\VV^\dagger)$ with the positive $c_\dagger$-cluster complex of Josuat--Verg\`es--Nadeau.
\end{lemma}

\begin{proof}
By \cref{lem:tilted-length}, $\VV^\dagger\simeq\operatorname{mod}\Gamma_\dagger$ is a representation-finite hereditary category of Dynkin type. Its rigid-module complex is the tilting complex: every basic rigid object admits a Bongartz completion, and the facets are the tilting objects with $r=\rank K_0(\VV^\dagger)$ indecomposable summands \cite[Section~3, especially Proposition~9]{Hubery2011}.

It remains only to match this standard hereditary model with the $c_\dagger$-convention of Josuat--Verg\`es--Nadeau. Let $S_1,\ldots,S_r$ be the simple objects of $\VV^\dagger$, put $d_i=\dim_K\End(S_i)$, and let $\rho_i$ be the unit simple roots in the standard geometric representation. The usual symmetrizable normalization
\[
  [S_i]\longmapsto \sqrt{2d_i}\,\rho_i
\]
identifies the symmetrized Euler form with the Coxeter geometric form and intertwines the simple reflections; compare the Cartan and root conventions in \cite[Sections~2.1--2.2]{HansonReading2025}. Under this normalization, mutual $\Ext^1$-orthogonality gives the nonnegative root-inner-product condition of \cite[Definition~A.1]{JVN2023}; conversely, within a complete exceptional sequence the same condition implies mutual $\Ext^1$-vanishing by \cite[Lemma~1.1]{BuanReitenThomas2011}.

Now let $\{X_1,\ldots,X_r\}$ be a rigid facet. Its summands can be ordered into a complete exceptional sequence \cite[Lemma~2.2(a)(ii)]{BuanReitenThomas2011}, and Hanson--Reading \cite[Theorem~2.2]{HansonReading2025}, applied to $\Gamma_\dagger$, gives
\[
  c_\dagger=t_{\dim X_1}\cdots t_{\dim X_r}.
\]
Conversely, the same correspondence sends a positive $c_\dagger$-cluster to a complete exceptional sequence, and the preceding pairing comparison gives mutual $\Ext^1$-vanishing. Thus the two complexes have the same facets. Every rigid face admits a Bongartz completion \cite[Section~3]{Hubery2011}, so the dimension-vector labelling identifies
\[
  \mathscr R(\VV^\dagger)\cong\Delta^+_{c_\dagger}.
\]
\end{proof}

\begin{proposition}[Dynkin block ball]\label{prop:Dynkin-ball}
The map
\begin{equation}\label{eq:tilt-map}
  \sigma(X)=\begin{cases}
    X,&\partial(X)\le0,\\
    X[-1],&\partial(X)>0
  \end{cases}
\end{equation}
identifies $K_\VV$ with the positive cluster complex of the hereditary Dynkin heart $\VV^\dagger$. Consequently $|K_\VV|$ is a closed ball and is contractible.
\end{proposition}

\begin{proof}
The map $\sigma$ is a bijection on indecomposable objects, and the four formulas in \cref{lem:tilted-hereditary} agree term-by-term with the same-side and mixed-side compatibility conditions of \cref{lem:defect-compatibility}. Thus $K_\VV$ is identified with $\mathscr R(\VV^\dagger)$. By \cref{lem:Dynkin-positive}, this is the positive finite-type $c_\dagger$-cluster complex, and \cite[Proposition~A.4]{JVN2023} identifies its realization with a closed $(r-1)$-ball.
\end{proof}

\subsection{Affine blocks and null-root restriction}
Fix a connected affine block $\VV$ and put $e=\cox(\VV)$. Let $\delta_\VV$ be its primitive positive null root and
\[
  L_\VV=\mathbb R\text{-span}\{\dim X:X\in\ind\VV\}\subseteq V.
\]

\begin{lemma}[Null-root restriction]\label{lem:null-restriction}
There is $a>0$ such that $\delta_\VV=a\delta$ inside the ambient Grothendieck group. Hence local defect is a positive scalar multiple of ambient defect and has the same sign.
\end{lemma}

\begin{proof}
By the radical computation \eqref{eq:block-radical} for connected blocks,
\[
  \rad(B|_{L_\VV})=L_\VV\cap\mathbb R\delta.
\]
Since $\VV$ is affine, its radical is generated by $\delta_\VV$, so $\delta_\VV=a\delta$ for some nonzero scalar $a$. The local null root is a positive linear combination of local simple classes, each having nonnegative ambient simple-root coordinates, hence $a>0$. Because the exact wide embedding preserves the nonsymmetric Euler form,
\[
  E_\VV(\delta_\VV,\dim X)=aE(\delta,\dim X),
\]
so local and ambient defects have the same sign.
\end{proof}

\begin{proposition}[Ambient--local affine comparison]\label{prop:affine-comparison}
Let $\Phi_\VV$ be the affine root system of the generalized Cartan lattice $K_0(\VV)$ from \cref{lem:block-cartan}. Then the identity on exceptional modules and the local half-orbit labelling give simplicial isomorphisms
\[
  K_\VV\cong\Delta_e^{\mathrm{exc}}(\VV)\cong\Delta_e^{\re}(\Phi_\VV).
\]
\end{proposition}

\begin{proof}
By \cref{lem:null-restriction}, local and ambient defects have the same sign. Because $\VV\hookrightarrow\HH$ is full, exact and extension-closed,
\[
  \Hom_\VV(X,Y)=\Hom_\HH(X,Y),\qquad
  \Ext^1_\VV(X,Y)=\Ext^1_\HH(X,Y).
\]
Thus the local and ambient compatibility tests of \cref{lem:defect-compatibility} coincide on $\ind\Exc(\VV)$, giving the first isomorphism. The second is the general realization form of \cref{thm:rectified}, applied to the connected tame hereditary category $\VV$; \cref{lem:block-cartan} identifies its local affine root system and Coxeter element with $\Phi_\VV$ and $e=\cox(\VV)$.
\end{proof}

\begin{theorem}[Principal fibre contractibility]\label{thm:principal-contractible}
For every $1\ne w\in P_c$, the principal fibre $K_w$ is contractible.
\end{theorem}

\begin{proof}
By \cref{thm:embedded-join}, $K_w$ is the join of the subcomplexes attached to the connected blocks of $\WW_w$. The subcomplexes attached to Dynkin blocks are contractible by \cref{prop:Dynkin-ball}; those attached to affine blocks are identified by \cref{prop:affine-comparison} with real affine complexes and are contractible by Theorem~\ref{thm:appendix-contractibility}. Since $w\ne1$, the join has a nonempty factor.
\end{proof}

\subsection{Coxeter--categorical component correspondence}\label{sec:complete-components}

For $w\in P_c$, set
\[
  T(w)=\{t\in T:t\le_Tw\},\qquad W(w)=\langle T(w)\rangle.
\]
Write the connected block decomposition
\[
  \WW_w=\VV_1\oplus\cdots\oplus\VV_s
\]
as in \cref{eq:block-decomp}, and put $e_i=\cox(\VV_i)$.

\begin{proposition}[Coxeter components of categorical blocks]\label{prop:component-categorical}
One has
\[
  T(w)=\bigsqcup_{i=1}^s T(e_i),
  \qquad
  W(w)=W(e_1)\times\cdots\times W(e_s).
\]
The factors $W(e_i)$ are precisely the connected components of Dyer's canonical Coxeter system of $W(w)$. Moreover,
\[
  w=e_1\cdots e_s=e_{\sigma(1)}\cdots e_{\sigma(s)}
\]
for every permutation $\sigma$, and $W(e_i)$ is finite exactly when $\VV_i$ is Dynkin and affine exactly when $\VV_i$ is affine.
\end{proposition}

\begin{proof}
Let $t\le_Tw$. The rank-one case of \cref{thm:HR} gives an exceptional module $X$ whose rank-one wide closure lies in $\WW_w$. The indecomposable $X$ belongs to a unique block $\VV_i$, so $t\le_Te_i$; the converse follows from $\VV_i\subseteq\WW_w$. Hence
\[
  T(w)=\bigsqcup_i T(e_i),
\]
and the corresponding root spans are pairwise orthogonal by the Hom--Ext orthogonality of distinct blocks.

By \cref{lem:block-cartan}, the local Weyl group of $\VV_i$ embeds in $W$ and equals $W(e_i)$: its simple reflections lie below $e_i$, while every reflection below $e_i$ comes from a rank-one exceptional wide subcategory of $\VV_i$. Thus the factors jointly generate $W(w)$. Their orthogonal root spans make them commute and act trivially on one another's spans. If $g_1\cdots g_s=1$ with $g_i\in W(e_i)$, restriction to the $i$th span shows that $g_i$ acts trivially there; faithfulness of the local geometric representation gives $g_i=1$. Hence the product is direct.

It remains to identify the canonical components. Let $S_w$ be Dyer's canonical simple-reflection set. By \cite[Theorem~3.3]{Dyer1990}, the reflections of $(W(w),S_w)$ are exactly $W(w)\cap T$. In the above orthogonal direct product, an ambient reflection lies in a unique factor: if it had nontrivial components in two orthogonal factors, its fixed-space codimension would be at least two. Hence
\[
  W(w)\cap T=\bigsqcup_i\bigl(W(e_i)\cap T\bigr).
\]
Set $S_i=S_w\cap W(e_i)$. Then $S_w=\bigsqcup_iS_i$, different $S_i$ commute, and projection to the $i$th factor shows that $S_i$ generates $W(e_i)$. If the canonical graph on some $S_i$ were disconnected, then $W(e_i)=A\times B$ for two nontrivial standard parabolic factors. Every reflection in a connected local simple system of $W(e_i)$ lies in $A$ or $B$, and both alternatives occur because that system generates $W(e_i)$. Reflections from different factors commute, contradicting connectedness. Thus each $S_i$ is connected, so the $W(e_i)$ are exactly the connected canonical components.

Finally, concatenating complete exceptional sequences from the blocks in any order and applying \cref{thm:HR} gives the product formula for $w$. The local Cartan form is positive definite for a Dynkin block and positive semidefinite of corank one for an affine block, giving the finite/affine characterization.
\end{proof}

\begin{corollary}[Finite component as a Dynkin join factor]\label{cor:finite-component-join}
If $e$ is an irreducible complete finite Coxeter component of $w$, then $e=e_j$ for a unique Dynkin block $\VV_j$. Writing
\[
  u=\prod_{i\ne j}e_i,
\]
one has
\[
  w=eu=ue,
  \qquad
  K_w=K_e*K_u
\]
as an equality of face sets inside $K_c$.
\end{corollary}

\begin{proof}
By \cref{prop:component-categorical}, the canonical components are the $W(e_i)$, and the finite ones are exactly those attached to Dynkin blocks; hence $e=e_j$ for a unique such block. The factors commute, so $w=eu=ue$. By \cref{thm:HR}, $\cox(\bigoplus_{i\ne j}\VV_i)=u$, and hence $\WW_u=\bigoplus_{i\ne j}\VV_i$. If $u=1$, then $\WW_u=0$ and $K_u=\{\varnothing\}$, so the join identity is immediate. If $u\ne1$, applying \cref{thm:embedded-join} to the decompositions of $\WW_w$ and $\WW_u$ gives $K_w=K_e*K_u$ inside the ambient complex $K_c$.
\end{proof}

\section{The affine fibre dichotomy}\label{sec:fibres}

For $b\in M_c$, define
\[
  \operatorname{Div}_{R,c}(b)=\{w\in P_c:w\preccurlyeq_Rb\},\qquad
  \operatorname{Max}_{R,c}(b)=\operatorname{Max}_{\preccurlyeq_R}\operatorname{Div}_{R,c}(b).
\]
Every interval right divisor of $b$ lies below a maximal one because $P_c$ has finite height. Consequently
\begin{equation}\label{eq:fibre-union}
  K_b=\bigcup_{w\in\operatorname{Max}_{R,c}(b)}K_w:
\end{equation}
indeed, a face $F\in K_b$ has $\omega(F)\in P_c$ below some maximal right divisor $w$, while the reverse inclusion follows from transitivity.

\begin{theorem}[Affine fibre dichotomy]\label{thm:fibre-dichotomy}
For every $1\ne b\in M_c$, exactly one of the following holds.
\begin{enumerate}[label=\textup{(\roman*)}]
\item The element $b$ has a greatest $P_c$-right divisor $d$, and $K_b=K_d$.
\item The element $b$ has no greatest $P_c$-right divisor. Then the maximal $P_c$-right divisors share a common nonidentity irreducible complete finite Coxeter component $e$. Writing each maximal divisor as $w=eu_w=u_we$ and setting
\[
  L_b=\bigcup_{w\in\operatorname{Max}_{R,c}(b)}K_{u_w},
\]
one has
\[
  K_b=K_e*L_b
\]
as an equality of ambient face sets.
\end{enumerate}
\end{theorem}

\begin{proof}
Use \cref{thm:interval-embedding} to regard $b$ in the positive interval-group image. If a greatest divisor exists, \eqref{eq:fibre-union} gives the first case. Otherwise rank two is excluded by \cref{rem:rank-two}; in rank at least three \cref{thm:common-core} gives a common complete finite component $e$. For each maximal $w$, \cref{cor:finite-component-join} gives $K_w=K_e*K_{u_w}$ as an ambient face-set equality. Since the same factor $K_e$ occurs inside the common ambient complex $K_c$, taking unions gives
\[
  K_b=\bigcup_w(K_e*K_{u_w})
      =K_e*\left(\bigcup_w K_{u_w}\right)
      =K_e*L_b.
\]
Moreover, $L_b$ is a simplicial subcomplex: a face in the union belongs to some $K_{u_w}$, and all of its subfaces belong to that same subcomplex.
\end{proof}

\begin{corollary}[Contractibility of homogeneous fibres]\label{thm:fibre-contractible}
For every $1\ne b\in M_c$, the homogeneous fibre $K_b$ is contractible.
\end{corollary}

\begin{proof}
By \cref{thm:fibre-dichotomy}, either $K_b=K_d$ for a principal fibre, which is contractible by \cref{thm:principal-contractible}, or $K_b=K_e*L_b$ with $e$ corresponding to a Dynkin block. In the latter case $K_e$ is nonempty and contractible by \cref{prop:Dynkin-ball}, so the join is contractible.
\end{proof}

\section{Weighted-face complexes and the Koszul resolution}\label{sec:weighted}

\subsection{A general weighted-face criterion}

Let $M=\bigsqcup_{d\ge0}M_d$ be a positively graded right-cancellative monoid with $M_0=\{1\}$ and let $A=\kk[M]$. Let $\mathcal K$ be an abstract simplicial complex.

\begin{definition}[Face weight]
A face weight is a map $\omega:\operatorname{Faces}(\mathcal K)\to M$ such that
\begin{equation}\label{eq:weight-axioms}
  \omega(\varnothing)=1,\qquad \deg\omega(F)=|F|,\qquad G\subseteq F\Rightarrow\omega(G)\preccurlyeq_R\omega(F).
\end{equation}
For $\gamma\in F$, right cancellation gives a unique $q_\gamma(F)$ with
\begin{equation}\label{eq:deletion-quotient}
  \omega(F)=q_\gamma(F)\omega(F\setminus\{\gamma\}).
\end{equation}
\end{definition}

Fix an arbitrary total order on $V(\mathcal K)$, used only for signs. Put, as a free left $A$-module,
\[
  P_q=A\otimes_\kk\kk\{F\in\mathcal K:|F|=q\}.
\]
For $F=\{\gamma_1<\cdots<\gamma_q\}$ define
\begin{equation}\label{eq:weighted-diff}
  d_q(a\otimes[F])=\sum_{i=1}^q(-1)^{i-1}a q_{\gamma_i}(F)\otimes[F\setminus\{\gamma_i\}].
\end{equation}

\begin{lemma}[Two-deletion identity]\label{lem:two-deletion}
For distinct $\gamma,\eta\in F$,
\[
  q_\gamma(F)q_\eta(F\setminus\{\gamma\})
  =q_\eta(F)q_\gamma(F\setminus\{\eta\}).
\]
\end{lemma}

\begin{proof}
Right multiply both sides by $\omega(F\setminus\{\gamma,\eta\})$. Both become $\omega(F)$; right cancellation gives the identity.
\end{proof}

\begin{proposition}
The maps \cref{eq:weighted-diff} satisfy $d_{q-1}d_q=0$.
\end{proposition}

\begin{proof}
Each unordered two-vertex deletion occurs twice. The coefficients agree by \cref{lem:two-deletion}, while the simplicial signs are opposite.
\end{proof}

Give the complex the fine monoid grading $\deg_M(a\otimes[F])=a\omega(F)$. For $b\in M$ define
\begin{equation}\label{eq:general-fibre}
  K_b=\{F\in\mathcal K:\omega(F)\preccurlyeq_Rb\}.
\end{equation}

\begin{proposition}[Fine-degree fibre identification]\label{prop:fine-strand}
For $b\ne1$, the monoid-degree-$b$ strand of the weighted complex is canonically the augmented simplicial chain complex of $K_b$.
\end{proposition}

\begin{proof}
A degree-$b$ basis vector labelled by $F$ exists exactly when $b=a_F\omega(F)$ for some $a_F\in M$, i.e. when $F\in K_b$. Right cancellation makes $a_F$ unique. If $G=F\setminus\{\gamma\}$, then
\[
  a_Fq_\gamma(F)\omega(G)=b=a_G\omega(G),
\]
so $a_Fq_\gamma(F)=a_G$. Under $a_F\otimes[F]\mapsto[F]$, the weighted differential is the ordinary augmented boundary.
\end{proof}

\begin{theorem}[Weighted-face criterion]\label{thm:weighted-criterion}
Assume \cref{eq:weight-axioms} and assume every nonidentity fibre $K_b$ is $\kk$-acyclic. Then \cref{eq:weighted-diff}, with the natural augmentation $P_0=A\to\kk$, is a minimal linear graded free resolution of the trivial left $A$-module. Hence $A$ is Koszul in the sense of \cref{conv:koszul}.
\end{theorem}

\begin{proof}
By \cref{prop:fine-strand}, every nonidentity fine-degree strand is an exact augmented simplicial chain complex; the identity strand is the one-dimensional augmentation strand. Their algebraic direct sum is therefore exact. A face generator in homological degree $q$ has internal degree $q$, and every deletion quotient has degree one, so the resolution is linear. All differential coefficients have positive degree. Tensoring with $\kk=A/A_+$ kills all differentials, proving graded minimality. Thus the augmentation module has a linear graded free resolution, which is precisely \cref{conv:koszul}; compare \cite[Section~2]{Priddy1970}.
\end{proof}

\begin{remark}
Neither local finiteness of $\mathcal K$ nor finite-dimensionality of the homogeneous pieces of $A$ is required by \cref{conv:koszul}. Each free module is an algebraic direct sum indexed by faces, and each fine-degree strand is the ordinary simplicial chain complex with finite chains.
\end{remark}

\subsection{Application to the affine dual braid monoid}
By \cref{thm:interval-embedding}, $M_c$ is positively graded and right-cancellative. By \cref{cor:weight-axioms}, the intrinsic weight satisfies \cref{eq:weight-axioms}. By \cref{thm:fibre-contractible}, every nonidentity fibre is contractible.

\begin{theorem}[Minimal exceptional-cluster resolution]\label{thm:minimal-resolution}
Choose an arbitrary total order on $V(K_c)=\ind\Exc(\HH)$, used only for signs. Put, as a free left $\Adual$-module,
\[
  P_q=\Adual\otimes_\kk\kk\Fr_q(K_c).
\]
For $F=\{X_1<\cdots<X_q\}$, let $q_{X_i}(F)\in M_c$ be the unique element satisfying
\[
  w(F)=q_{X_i}(F)w(F\setminus\{X_i\}).
\]
By \cref{lem:interval-degree}, $q_{X_i}(F)$ has degree one and is therefore a reflection generator. Under the embedding \cref{thm:interval-embedding}, it is the explicit group quotient
\[
  q_{X_i}(F)=w(F)w(F\setminus\{X_i\})^{-1}\in T.
\]
Define
\[
  d_q(a\otimes[F])=\sum_{i=1}^q(-1)^{i-1}a q_{X_i}(F)\otimes[F\setminus\{X_i\}].
\]
Then
\[
  0\longrightarrow P_n\longrightarrow\cdots\longrightarrow P_1\longrightarrow P_0\longrightarrow\kk\longrightarrow0
\]
is a minimal linear graded free resolution of the trivial $\Adual$-module.
\end{theorem}

\begin{proof}
Apply \cref{thm:weighted-criterion} to the intrinsic exceptional face weight. By \cref{thm:face-map}, every facet has $n$ vertices, so $P_q=0$ for $q>n$.
\end{proof}

\subsection{Koszulity, Tor, and projective dimension}

\begin{corollary}[Homological consequences]\label{cor:homological}
For every field $\kk$,
\[
  \Adual=\kk[M([1,c]_T)]\text{ is Koszul},
  \qquad
  \Tor_q^{\Adual}(\kk,\kk)\cong\kk\Fr_q(K_c)
\]
with $\Tor_q$ concentrated in internal degree $q$, and
\[
  \pd_{\Adual}\kk=n.
\]
\end{corollary}

\begin{proof}
The resolution of \cref{thm:minimal-resolution} is minimal and linear, hence gives Koszulity. Regard $\kk=\Adual/A_{c,+}$ as the augmentation right module. Since every differential coefficient lies in $A_{c,+}$, the differential on $\kk\otimes_{\Adual}P_\bullet$ vanishes, and
\[
  \kk\otimes_{\Adual}P_q\cong\kk\Fr_q(K_c),
\]
which yields the displayed Tor basis. The resolution has length at most $n$, while an $n$-vertex facet gives nonzero $\Tor_n$, so $\pd_{\Adual}\kk=n$.
\end{proof}

\subsection{The reducible case and the main theorem}

\begin{theorem}[Affine exceptional-cluster Koszul resolution]\label{thm:main}
Let $(W,S)$ be a finite direct product of irreducible crystallographic affine Coxeter systems, let $c$ be a Coxeter element, and let $\kk$ be any field. Then $\kk[M([1,c]_T)]$ is Koszul, and the modules and differential of \cref{thm:minimal-resolution} define an explicit minimal linear resolution of the trivial left module. Moreover,
\[
  \Tor_q^{\Adual}(\kk,\kk)\cong\kk\Fr_q(\Delta_c^{\mathrm{exc}}(\HH)),
  \qquad \pd_{\Adual}\kk=|S|.
\]
Via \cref{thm:rectified}, the same Tor basis may equivalently be labelled by real Reading--Stella faces.
\end{theorem}

\begin{proof}
The preceding argument proves the irreducible affine case, including rank two by \cref{rem:rank-two}. For a finite product of irreducible affine systems, absolute intervals and interval monoids decompose as direct products, while the rectified exceptional complex is the simplicial join of component complexes. The intrinsic face weight is the componentwise product. Tensoring the component resolutions over the field $\kk$ preserves exactness, linearity and minimality, and homological degrees add.
\end{proof}

\appendix
\section{Contractibility of the real affine Reading--Stella complex}\label{app:topology}

This appendix proves the real affine cluster contractibility used in \cref{thm:principal-contractible}. The argument uses uniqueness of cluster expansions, deletion of the relative interior of the $\delta$-star support, and uniqueness of the projective accumulation direction $[\delta]$. The comparison with the abstract complex is made through a coefficient-star good cover.

\begin{theorem}[Contractibility of the real affine cluster complex]\label{thm:appendix-contractibility}
Let $\Phi$ be an irreducible symmetrizable affine root system of rank $r$, let $e$ be a Coxeter element, and let $\Delta_e^{\re}(\Phi)$ be the real Reading--Stella $e$-cluster complex. Then
\[
  |\Delta_e^{\re}(\Phi)|\simeq *.
\]
\end{theorem}

Let $V$ be the real root space. Let $\operatorname{Fan}_e(\Phi)$ be the full Reading--Stella fan of nonnegative spans of $e$-cluster faces and $\operatorname{Fan}_e^{\re}(\Phi)$ its real subfan. Unique cluster expansions and completeness of the full simplicial fan are \cite[Theorems~6.2 and 6.4]{ReadingStella2020}.

Let $U^e\subset V$ be the associated codimension-one hyperplane. Let $\Upsilon_e=\Phi\cap U^e$ and let $I_e\subset U^e$ be the Reading--Stella imaginary cone, i.e. the nonnegative span of the positive roots in $\Upsilon_e$, equivalently the cone generated by its canonical simple system. Let $\Lambda_e^{\re}$ be the distinguished finite set of regular real roots compatible with $\delta$, and put $\Lambda_e=\Lambda_e^{\re}\sqcup\{\delta\}$. We use the $\Lambda_e$-supported expansion on $I_e$ from \cite[Proposition~6.8]{ReadingStella2020}, the characterization of real roots compatible with $\delta$ from \cite[Proposition~5.6]{ReadingStella2020}, the non-component-full property of pairwise compatible regular supports from \cite[Lemma~5.10]{ReadingStella2020}, and the star-support identity from \cite[Proposition~6.13(2)]{ReadingStella2020}.

\begin{lemma}[Imaginary deletion]\label{lem:imaginary-deletion}
The support of the real subfan is
\begin{equation}\label{eq:real-support}
  \Supp\operatorname{Fan}_e^{\re}(\Phi)=V\setminus\relint_{U^e}(I_e).
\end{equation}
\end{lemma}

\begin{proof}
For $v\in V$, let $m_\delta(v)$ be the coefficient of $\delta$ in its unique cluster expansion. We claim
\begin{equation}\label{eq:delta-coefficient}
  m_\delta(v)>0\iff v\in\relint_{U^e}(I_e).
\end{equation}
If $m_\delta(v)>0$, write $v=a\delta+y$ with $a>0$ and all roots in the support of $y$ compatible with $\delta$. The supporting cone lies in the star of the $\delta$-ray, hence $y\in I_e$ by \cite[Proposition~6.13(2)]{ReadingStella2020}. Since $\delta\in\relint I_e$, the elementary relative-interior property of convex cones gives $v\in\relint I_e$.

Conversely, let $v\in\relint I_e$. Its distinguished expansion is supported in $\Lambda_e$ by \cite[Proposition~6.8]{ReadingStella2020}. If the $\delta$-coefficient were zero, the support would be a pairwise compatible subset $R\subseteq\Lambda_e^{\re}$. By \cite[Lemma~5.10]{ReadingStella2020}, in every irreducible component of $\Upsilon_e$ the support of $R$ misses at least one canonical simple root. Therefore every nonnegative combination supported on $R$ lies on the boundary of the canonical-simple-root cross-section of $I_e$, contradicting $v\in\relint I_e$. This proves \cref{eq:delta-coefficient}.

By \cite[Definition~3.1]{ReadingStella2020}, $\delta$ is the unique imaginary vertex. Hence a vector lies in a cone supported entirely on real vertices exactly when its $\delta$-coefficient is zero. Completeness of the full fan and \cref{eq:delta-coefficient} give \cref{eq:real-support}.
\end{proof}

Intersect with a unit sphere $S(V)\cong S^{r-1}$ and set
\[
  D_e=S(V)\cap I_e,\qquad
  Y_e=S(V)\cap\Supp\operatorname{Fan}_e^{\re}(\Phi).
\]

\begin{lemma}[Spherical complement]\label{lem:spherical-complement}
The space $Y_e$ is contractible.
\end{lemma}

\begin{proof}
If $r=2$, then $U^e$ is one-dimensional and contains $\delta$, so $U^e=\mathbb R\delta$. Since $I_e\subset U^e$ is a nonzero pointed cone containing $\delta$ in its relative interior, necessarily $I_e=\mathbb R_{\ge0}\delta$. Hence $D_e=\{[\delta]\}\subset S^1$ and, by \cref{lem:imaginary-deletion},
\[
  Y_e=S^1\setminus\{[\delta]\}\cong\mathbb R.
\]
Assume from now on that $r\ge3$.

Reading--Stella identify $I_e$ as a closed pointed polyhedral cone whose extreme rays are generated by the canonical simple roots $\Xi_e$ of $\Upsilon_e$ and whose relative interior contains $\delta$ \cite[discussion preceding Proposition~6.6 and Proposition~6.13(2)]{ReadingStella2020}. It is full-dimensional in $U^e$: an imaginary $e$-cluster has $r-1$ vertices and is a $\mathbb Z$-basis of the lattice $Q\cap U^e$ by \cite[Proposition~5.14(5)--(6)]{ReadingStella2020}, and its cone is contained in $I_e$.

Choose a linear functional $\lambda>0$ on $I_e\setminus\{0\}$. The affine section
\[
  P_e=\{v\in I_e:\lambda(v)=1\}
\]
is a compact convex polytope with nonempty interior in an $(r-2)$-dimensional affine hyperplane. Radial normalization is a homeomorphism $P_e\to D_e$, so $D_e$ is a proper full-dimensional spherically convex polyhedral closed $(r-2)$-ball in the equator $E=S(V)\cap U^e\cong S^{r-2}$. In particular $\partial D_e\subset E$ is a polyhedral embedded $(r-3)$-sphere. Standard PL regular-neighborhood theory makes this codimension-one polyhedral embedding locally flat and bicollared; see, for example, \cite[Chapters~3--4]{RourkeSanderson1982}. Thus, for $r\ge4$, Brown's generalized Schoenflies theorem applies to the middle sphere of such a bicollar and implies that the two closures of $E\setminus\partial D_e$ are topological $(r-2)$-balls \cite{Brown1960}. For $r=3$ the same conclusion is immediate because the pair $(E,D_e)$ is a circle together with a closed arc. A boundary homeomorphism between either of these balls and a standard hemisphere extends across the other ball by the Alexander trick, so the two ball homeomorphisms may be chosen to agree on their common boundary. Thus in every case $r\ge3$ there is an ambient homeomorphism of the pair $(E,D_e)$ with the pair consisting of $S^{r-2}$ and a standard closed hemisphere. Extending this homeomorphism over the two cones of the suspension $S^{r-1}=\Sigma E$ reduces the problem to the standard hemisphere
\[
  D_+=\{(x_1,x',0)\in S^{r-1}:x_1\ge0\}.
\]
By \cref{lem:imaginary-deletion}, it is enough to contract $S^{r-1}\setminus\relint D_+$.

Put $p=(-1,0,\ldots,0)$ and define
\begin{equation}\label{eq:hemisphere-contraction}
  H_t(x)=\frac{(1-t)x+tp}{\|(1-t)x+tp\|},\qquad0\le t\le1.
\end{equation}
The denominator can vanish only at $x=-p=(1,0,\ldots,0)$, which belongs to $\relint D_+$ and is therefore absent from the domain. If the normal coordinate $x_r$ is nonzero, then for $t<1$ the normal coordinate of $(1-t)x+tp$ is also nonzero, so the path cannot enter the equatorial set $D_+$. If $x_r=0$, then $x_1\le0$ because $x\notin\relint D_+$, and the first coordinate $(1-t)x_1-t$ remains nonpositive. Hence the homotopy stays in $S^{r-1}\setminus\relint D_+$ and contracts it to $p$.
\end{proof}

Let $\tau_e$ be the Reading--Stella piecewise permutation associated with $e$.

\begin{lemma}[Unique projective accumulation]\label{lem:unique-accumulation}
Every sequence of distinct real cluster roots that leaves every finite subset converges projectively to the positive null-root ray $[\delta]$. Moreover, every transjective real vertex has finite link.
\end{lemma}

\begin{proof}
Reading--Stella prove
\begin{equation}\label{eq:gamma-delta}
  (e-1)\gamma_e=\delta
\end{equation}
and that $e$ has finite order on $U^e$ \cite[Proposition~2.10]{ReadingStella2020}. Choose $M>0$ with $e^M|_{U^e}=\mathrm{id}$. If $\beta=a\gamma_e+u$ with $u\in U^e$, then
\begin{equation}\label{eq:Jordan-tail}
  e^{qM}\beta=\beta+qMa\delta,\qquad q\in\mathbb Z.
\end{equation}

By \cite[Proposition~3.12]{ReadingStella2020}, there are finitely many finite $\tau_e$-orbits and finitely many infinite ones. A sequence leaving every finite subset is therefore eventually on the infinite orbits. On any one of their tails, $\tau_e(\alpha)=e\alpha$ except at the negative-simple splice and its paired root. Far enough out, the tail agrees term-by-term with iterates of the linear map $e$. The coefficient $a$ above is nonzero, otherwise finite order on $U^e$ would make the orbit finite. \cref{eq:Jordan-tail} shows that the normalized directions tend to the line $\mathbb R\delta$; all but finitely many roots on either tail are positive, so the projective limit is the positive ray $[\delta]$. Since there are only finitely many tails and all have the same limit, the original sequence converges to $[\delta]$.

A power of $\tau_e$ sends every transjective vertex to a negative simple root. The link of a negative simple root is the cluster complex of a proper standard parabolic subsystem \cite[Proposition~5.4(3)]{ReadingStella2020}, which is finite in irreducible affine type. Hence every transjective real vertex has finite link.
\end{proof}

\begin{lemma}[Bounded-support extraction]\label{lem:bounded-extraction}
Let $(C_k)$ be a sequence of finite subsets of a set, with $|C_k|\le r$. After passage to a subsequence there are a fixed finite set $R$ and sets $B_k$ such that
\[
  C_k=R\sqcup B_k
\]
for every $k$, and every element outside $R$ occurs in only finitely many of the $B_k$.
\end{lemma}

\begin{proof}
If some element occurs in infinitely many $C_k$, pass to a subsequence on which it occurs always and place it in $R$. Repeat. Each repetition increases $|R|$ by one, while every $C_k$ has at most $r$ elements, so the process stops after at most $r$ steps. At termination no element outside $R$ occurs infinitely often. Removing $R$ from the surviving supports gives the required $B_k$.
\end{proof}

For a real vertex $\alpha$, let $m_\alpha(x)$ be its coefficient in the unique cluster expansion of $x\in Y_e$ and define
\[
  U_\alpha=\{x\in Y_e:m_\alpha(x)>0\}.
\]

\begin{lemma}[Openness of coefficient stars]\label{lem:star-open}
Every $U_\alpha$ is open in $Y_e$.
\end{lemma}

\begin{proof}
Since $Y_e$ is metrizable, it suffices to prove sequential openness. Suppose $x_k\to x\in U_\alpha$ but $x_k\notin U_\alpha$. Let $C(y)=\{\beta:m_\beta(y)>0\}$. Each cluster has at most $r$ vertices. Apply \cref{lem:bounded-extraction} and pass to a subsequence with
\[
  C(x_k)=R\sqcup B_k,
\]
where $R$ is fixed and every vertex outside $R$ occurs in only finitely many $B_k$.

If $B_k$ is eventually empty, then $x_k\in\Cone(R)$ eventually. The cone is closed, so uniqueness of expansion gives $C(x)\subseteq R$, forcing $\alpha\in R\subseteq C(x_k)$, a contradiction.

Assume $B_k$ is nonempty. The normalized directions of all vertices in $B_k$ converge uniformly to $[\delta]$: otherwise there are $\varepsilon>0$ and choices $\beta_k\in B_k$ staying at projective distance at least $\varepsilon$ from $[\delta]$. Every individual vertex occurs only finitely often among the $\beta_k$, so $(\beta_k)$ leaves every finite subset, contradicting \cref{lem:unique-accumulation}. Every $\rho\in R$ is compatible with infinitely many distinct varying vertices and hence cannot be transjective, because transjective links are finite. Thus $R\subseteq\Lambda_e^{\re}$, and \cite[Proposition~5.6]{ReadingStella2020} makes $R\cup\{\delta\}$ compatible.

Choose a linear functional $L$ strictly positive on every nonzero vector in the positive root cone. After finitely many terms every varying root is positive, and the fixed roots are positive regular roots. Write
\[
  x_k=\sum_{\beta\in C(x_k)}a_{\beta,k}\beta,\qquad a_{\beta,k}>0,
\]
and set
\[
  \lambda_{\beta,k}=\frac{a_{\beta,k}L(\beta)}{L(x_k)}.
\]
Then $\sum_\beta\lambda_{\beta,k}=1$ and
\[
  \frac{x_k}{L(x_k)}
  =\sum_{\rho\in R}\lambda_{\rho,k}\frac{\rho}{L(\rho)}
   +\sum_{\beta\in B_k}\lambda_{\beta,k}\frac{\beta}{L(\beta)}.
\]
Pass to a further subsequence so that every fixed coefficient $\lambda_{\rho,k}$ and the total varying weight
\[
  \lambda_{B,k}=\sum_{\beta\in B_k}\lambda_{\beta,k}
\]
converge. Uniform projective convergence of the vertices in $B_k$, together with positivity of $L$, gives uniform convergence $\beta/L(\beta)\to\delta/L(\delta)$ on $B_k$. Therefore the varying sum converges to $\lambda_B\delta/L(\delta)$, and since $x_k\to x$ with $L(x)>0$ we obtain
\[
  \frac{x}{L(x)}
  =\sum_{\rho\in R}\lambda_\rho\frac{\rho}{L(\rho)}
    +\lambda_B\frac{\delta}{L(\delta)}
  \in\Cone(R\cup\{\delta\}).
\]
Hence $x\in\Cone(R\cup\{\delta\})$. Since $R\cup\{\delta\}$ is compatible, it is contained in a cluster and therefore spans a face cone of the full Reading--Stella fan \cite[Theorem~6.4]{ReadingStella2020}. Since $x\in Y_e$, \cref{eq:delta-coefficient} forces its $\delta$-coefficient to vanish. Thus $C(x)\subseteq R$, and again $\alpha\in R\subseteq C(x_k)$, a contradiction.
\end{proof}

\begin{proposition}[Coefficient-star good cover]\label{prop:good-cover}
The sets $U_\alpha$ form an open numerable good cover of $Y_e$, and its nerve is the abstract complex $\Delta_e^{\re}(\Phi)$.
\end{proposition}

\begin{proof}
Every nonzero vector has nonempty cluster support, so the $U_\alpha$ cover $Y_e$. For a finite set $F$ of real vertices, put
\[
  U_F=\bigcap_{\alpha\in F}U_\alpha,\qquad q_F=\sum_{\alpha\in F}\alpha.
\]
Uniqueness of expansion gives
\begin{equation}\label{eq:nerve-faces}
  U_F\ne\varnothing\iff F\in\Delta_e^{\re}(\Phi).
\end{equation}
Indeed, if $x\in U_F$, all vertices of $F$ occur in its cluster support; conversely, if $F$ is a real face then $q_F/\|q_F\|\in U_F$.

When $U_F\ne\varnothing$, define
\begin{equation}\label{eq:star-contraction}
  H_t(x)=\frac{(1-t)x+tq_F}{\|(1-t)x+tq_F\|},\qquad0\le t\le1.
\end{equation}
The set $F$ lies in the support of $x$, so $x$ and $q_F$ lie in the same pointed cluster cone. The denominator is nonzero and every coefficient indexed by $F$ remains positive. Thus \cref{eq:star-contraction} contracts $U_F$ to $q_F/\|q_F\|$ inside $U_F$. Metric spaces are paracompact, so the open good cover is numerable.
\end{proof}

\begin{proof}[Proof of \cref{thm:appendix-contractibility}]
By \cref{prop:good-cover}, the numerable nerve theorem \cite[Corollary~4G.3]{Hatcher2002} gives
\[
  |\Delta_e^{\re}(\Phi)|\simeq Y_e.
\]
The target is contractible by \cref{lem:spherical-complement}.
\end{proof}

\begin{remark}[Weak and fan-support topologies]
The additional topological argument above is needed because the infinite abstract realization and the support of the real cluster fan carry different natural topologies: the weak topology and the Euclidean subspace topology, respectively. The coefficient-star cover provides the required homotopy comparison between them; the only affine limiting behavior involved is the accumulation of real roots at the projective direction $[\delta]$.
\end{remark}

\bibliographystyle{amsplain}
\bibliography{references}

\end{document}